\documentclass[leqno,11pt]{amsart}

\usepackage[letterpaper,margin=1in]{geometry}
\usepackage{arydshln}
\usepackage[all]{xy} 
\usepackage{amsmath, amssymb, amsfonts, latexsym, mdwlist, amsthm, amscd,mathabx,braket}
\usepackage{subfig}
\usepackage{graphicx}
\usepackage{wrapfig}
\usepackage{etex}
\usepackage{tikz-cd}
\usepackage[colorlinks=true, citecolor=blue, urlcolor=blue, linkcolor=blue, pagebackref]{hyperref}

\usepackage{tikz}
\usetikzlibrary{calc,trees,positioning,arrows,chains,shapes.geometric,%
    decorations.pathreplacing,decorations.pathmorphing,shapes,%
    matrix,shapes.symbols}

\tikzset{
>=stealth',
  punktchain/.style={
    rectangle,
    rounded corners,
    draw=black, thick,
    minimum height=3em,
    text centered,
    on chain},
  line/.style={draw, thick, <-},
  eLement/.style={
    tape,
    top color=white,
    bottom color=blue!50!black!60!,
    minimum width=8em,
    draw=blue!40!black!90, very thick,
    text width=10em,
    minimum height=3.5em,
    text centered,
    on chain},
  every join/.style={->, thick,shorten >=1pt},
  decoration={brace},
  tuborg/.style={decorate},
  tubnode/.style={midway, right=2pt},
}

\usepackage{paralist}
\usepackage{enumitem} 
\setdefaultenum{(1)}{(a)}{(i)}{}
\setlist[enumerate,1]{label={\upshape(\arabic*)}}
\setlist[enumerate,2]{label={\upshape(\alph*)},ref=\theenumi\upshape(\alph*)}
\setlist[enumerate,3]{label={\upshape(\roman*)},ref=\theenumi\theenumii\upshape(\roman*)}
\usepackage[capitalize]{cleveref}
\crefname{Prop}{Proposition}{Propositions}
\crefname{Thm}{Theorem}{Theorems}
\crefname{Lem}{Lemma}{Lemmas}
\crefname{enumi}{Case}{Cases}

\def\dim{\mathop{\mathrm{dim}}\nolimits}

\def\min{\mathop{\mathrm{min}}\nolimits}

\def\MG13{\ensuremath{{\mathcal M}_{\Gamma_1(3)}}}
\def\tildeMG13{\ensuremath{\widetilde{\mathcal M}_{\Gamma_1(3)}}}

\newcommand\TFILTB[3]{%
\xymatrix@=1pc{
{0 = {#1}_0} \ar[rr]&&
{{#1}_1} \ar[rr]\ar[ld] &&
{{#1}_2} \ar[r]\ar[ld] &
{\cdots} \ar[r] & { {#1}_{#3-1}} \ar[rr] &&
{{#1}_{#3} = {#1}} \ar[ld]
\\
& *{{#2}_1} \ar@{.>}[ul] &&
{{#2}_2} \ar@{.>}[ul] & &&&
{{#2}_{{#3}}} \ar@{.>}[ul]
}}

\makeatletter
\newtheorem*{rep@theorem}{\rep@title}
\newcommand{\newreptheorem}[2]{%
\newenvironment{rep#1}[1]{%
 \def\rep@title{#2 \ref{##1}}%
 \begin{rep@theorem}}%
 {\end{rep@theorem}}}
\makeatother

\newtheorem{Thm}{Theorem}[section]
\newreptheorem{Thm}{Theorem}
\newtheorem{Prop}[Thm]{Proposition}

\newtheorem{Lem}[Thm]{Lemma}

\newtheorem{Cor}[Thm]{Corollary}
\newreptheorem{Cor}{Corollary}

\newreptheorem{Con}{Conjecture}

\newtheorem*{theorem*}{Theorem}
\newtheorem*{lemma*}{Lemma}
\newtheorem*{proposition*}{Proposition}
\newtheorem*{conjecture*}{Conjecture}
\newtheorem*{corollary*}{Corollary}
\newtheorem*{problem*}{Problem}

\newtheorem{Thm-int}{Theorem}

\theoremstyle{definition}
\newtheorem{Def-s}[Thm]{Definition}
\newtheorem{Def}[Thm]{Definition}
\newtheorem{Rem}[Thm]{Remark}

\newtheorem{Prob}[Thm]{Problem}
\newtheorem{Ex}[Thm]{Example}

\newcommand{\ignore}[1]{}
\begin{document}

\title{Rigid Schubert classes in partial flag varieties}
\author{Yuxiang Liu ${}^{1}$, Artan Sheshmani${}^{1,2,3}$ and  Shing-Tung Yau$^{2,4}$}

\address{${}^1$ Beijing Institute of Mathematical Sciences and Applications, No. 544, Hefangkou Village, Huaibei Town, Huairou District, Beijing 101408}

\address{${}^2$  Massachusetts Institute of Technology, IAiFi Institute, 77 Massachusetts Ave, 26-555. Cambridge, MA 02139, artan@mit.edu}

\address{${}^3$ National Research University Higher School of Economics, Russian Federation, Laboratory of Mirror Symmetry, NRU HSE, 6 Usacheva str.,Moscow, Russia, 119048}
\address{${}^4$ Yau Mathematical Sciences Center, Tsinghua University, Haidian District, Beijing, China}


\begin{abstract}
A Schubert class is called {\em rigid} if it can only be represented by Schubert varieties. The rigid Schubert classes have been classified in Grassmannians \cite{Coskun2011RigidAN} and orthogonal Grassmannians \cite{YL}. In this paper, we study the rigidity problem in partial flag varieties of classical type. In particular, we give numerical conditions that ensure a Schubert class is rigid.
    
\end{abstract}

\maketitle


\tableofcontents

\section{Introduction}
The Schubert varieties form a distinguished class of subvarieties in a rational homogeneous space. In the basic example of Grassmannians, those subvarieties have a great geometric interpretation: they can be depicted by the rank conditions imposed by a partial flag. One may then ask whether those rank conditions are preserved under rational equivalence. In other words, given a Schubert class in a Grassmannian, can one expect the same rank conditions among all the representatives of the class? This question was answered in \cite{YL}.

In this paper, we study a similar problem in the setting of partial flag varieties of classical type. In particular, we give numerical conditions that ensure a Schubert class is rigid.
\subsection{Partial flag varieties (of type $A$)}
Let $V$ be a complex vector space of dimension $n$. The partial flag variety $F=F(d_1,...,d_k;n)$ parametrizes all $k$-step partial flags $\Lambda_1\subset...\subset \Lambda_k$, where $\Lambda_i$ are subspaces of $V$ with dimension $d_i$, $1\leq i\leq k$. Let $F_{a_1}^{\alpha_1}\subset...\subset F_{a_{d_k}}^{\alpha_{d_k}}$
be a partial flag of linear subspaces of $V$, where the lower indices indicate the vector space dimension of the flag elements and the upper indices are integers between $1$ and $k$ such that $\#\{i|\alpha_i\leq t\}=d_t$ for all $1\leq t\leq k$. Set 
$$\mu_{i,j}:=\#\{s|a_s\leq a_i,\alpha_s\leq j\}.$$ 
The corresponding Schubert variety is defined to be the following locus 
$$\Sigma_{a^\alpha}(F_\bullet):=\{(\Lambda_1,...,\Lambda_k)|\dim(F_{a_i}\cap\Lambda_j)\geq \mu_{i,j}, \text{ for }1\leq i\leq d_k, \alpha_i\leq j\leq k\}.$$

Its rational equivalence class is independent of the choice of $F_\bullet$ and is denoted by $\sigma_{a^\alpha}\in A^*(F;\mathbb{Z})$. Given a Schubert class $\sigma_{a^\alpha}$, for each $1\leq i\leq d_k$ we ask the following question:
\begin{Prob}\label{problem}
Let $X$ be a subvariety of $F(d_1,...,d_k;n)$ with class $[X]=\sigma_{a^\alpha}$. Does there exist a vector subspace $F_{a_i}\subset V$ of dimension $a_i$ such that 
$$\dim(F_{a_i}\cap \Lambda_j)\geq\mu_{i,j}, \text{ for all } (\Lambda_1,...,\Lambda_k)\in X\text{ ? }$$
\end{Prob}
If the answer is true for all the representatives of $\sigma_{a^\alpha}$, then we call the corresponding sub-index $a_i$ {\em rigid}. For Schubert varieties, such vector subspaces always exist by definition. It is also easy to check that those vector subspaces are unique if and only if $\text{ either }a_{i}< a_{i+1}-1\text{ or }\alpha_{i}<\alpha_{i+1}.$ The sub-indices that satisfy one of the two inequalities are called essential.

In the case of Grassmannians (i.e. the partial flag varieties with $k=1$), the rigid sub-indices have been classified in \cite{YL},

For a general flag variety $F(d_1,...,d_k;n)$, there are canonical projections 
$$\pi_t:F(d_1,...,d_k;n)\rightarrow G(d_t,n),\ t=1,...,k.$$
The push-forward of Schubert classes under those canonical projections are well-defined. Our first main result is the answer to Problem \ref{problem}.

\begin{Thm}
Let $\sigma_{a^\alpha}$ be a Schubert class in $F(d_1,...,d_k;n)$. Let $a_i$ be an essential sub-index. For every representative $X$ of $\sigma_{a^\alpha}$, there exists a linear subspace $F_{a_i}$ of dimension $a_i$ such that 
$$\dim(F_{a_i}\cap \Lambda_j)\geq\mu_{i,j}, \text{ for all } (\Lambda_1,...,\Lambda_k)\in X$$
if and only if $a_i$ is rigid with respect to the Schubert class $(\pi_t)_*(\sigma_{a^\alpha})$ in $G(d_t,n)$ for some $1\leq t\leq k$.
\end{Thm}

Combining this result with the classification of rigid sub-indices in Grassmannians, we obtained the following numerical conditions that ensure a sub-index is rigid:
\begin{Cor}
Given a Schubert index $a^\alpha=(a_1^{\alpha_1},...,a_{d_k}^{\alpha_{d_k}})$ in $F(d_1,...,d_k;n)$. Set $a_{d_k+1}=\infty$ and $a_0=0$. An essential sub-index $a_i$ is rigid if and only if one of the following holds:
\begin{enumerate}
\item $a_{i+1}-a_i\geq 3$;
\item $a_{i+1}-a_i=2$ and 
\begin{enumerate}
\item either $a_i-a_{i-1}=1$; or
\item $\alpha_i<\alpha_{i+1}$
\end{enumerate}
\item $a_{i+1}-a_i=1$, and 
\begin{enumerate}
\item either $a_{i+2}-a_i\geq 3$; or
\item $\alpha_{i+2}>\alpha_i$; or
\item $a_{i-1}=a_i-1$ and $\alpha_{i-1}<\alpha_{i+1}$.
\end{enumerate}
\end{enumerate}
\end{Cor}

As an application, we characterize the rigid Schubert classes in $F(d_1,...,d_k;n)$. Given a representative $X$ of a Schubert class $\sigma_{a^\alpha}$, if all essential sub-indices are determined, then the previous result gives a unique set of linear subspaces. If furthermore these linear subspaces form a partial flag, then $X$ must be the Schubert variety defined by this partial flag. We introduce a relation on the set of all essential sub-indices that reflects the compatibility between two sub-indices:

\begin{Def}
Let $\sigma_{a^\alpha}$ be a Schubert class in $F(d_1,...,d_k;n)$. We define a relation `$\rightarrow$' between two sub-indices: $a_i\rightarrow a_j$ if $i<j$ and $a_j$ is essential in $(\pi_t)_*(\sigma_{a^\alpha})$ for some $t\geq \min(\alpha_i,\alpha_j)$. This relation extends to a strict partial order (which we also denote by `$\rightarrow$') on the set of essential sub-indices by transitivity.
\end{Def}

\begin{Thm}\label{rigid in f}
A Schubert class $\sigma_{a^\alpha}\in A(F(d_1,...,d_k;n))$ is rigid if and only if all essential sub-indices are rigid and the set of all essential sub-indices is strict totally ordered under the relation `$\rightarrow$'.
\end{Thm}

\subsection{Orthogonal partial flag varieties}Our next result is in orthogonal partial flag varieties which generalize orthogonal Grassmannians. Let $V$ be a complex vector space of dimension $n$, and let $q$ be a non-degenerate symmetric bilinear form on $V$. The orthogonal partial flag variety $OF(d_1,...,d_k;n)$ parametrizes all $k$-step partial flags $(\Lambda_1,...,\Lambda_k)$, where $\Lambda_i\subset \Lambda_{i+1}$ and $\Lambda_i$ are isotropic subspaces of dimension $d_i$, unless $n=2d_k$, in which case the parameter space has two irreducible components and we let $OF(d_1,...,d_k;n)$ denote one of the components.

The Schubert varieties in $OF(d_1,...,d_k;n)$ can be defined similarly by assigning each Schubert index in the orthogonal Grassmannian $OG(d_k,n)$ with an "upper index".
A Schubert index for $OF(d_1,...,d_k;n)$ consists of two sequences $$1\leq a_1^{\alpha_1}<...<a^{\alpha_s}_s\leq \left[\frac{n}{2}\right],$$
$$0\leq b^{\beta_1}_1<...<b^{\beta_{d_k-s}}_{d_k-s}\leq\left[\frac{n}{2}\right]-1,$$
such that there is no $i$ and $j$ such that $a_i-b_j=1$, and an arrangement of the elements in $a$ and $b$ into $k$ blocks so that the $i$-block has $d_i-d_{i-1}$ elements. The upper index $\alpha_i$ indicates the block number of $a_i$ and the upper index $\beta_j$ indicates the block number of $b_j$. When $n=2d_k$, we further require $s$ and $d_k$ have the same parity.

Fix a flag of isotropic subspaces $F_1\subset ...\subset F_{\left[n/2\right]}$ in $V$. Let $F_i^\perp$ be the orthogonal complement of $F_i$ with respect to $q$, $1\leq i\leq \left[n/2\right]$. When $n$ is even, the orthogonal complement of $F_{n/2-1}$ is a union of two maximal isotropic subspaces, one is $F_{n/2}$ and the other one belongs to the different irreducible component than $F_{n/2}$. By abuse of notation, we denote also by $F_{n/2-1}^\perp$ the maximal isotropic subspace in the orthogonal complement of $F_{n/2-1}$ other than $F_{n/2}$. 

Given a Schubert index $(a^\alpha,b^\beta)$ for $OF(d_1,...,d_k;n)$, we set
$$\mu_{i,t}:=\#\{c|a_c\leq a_i,\alpha_c\leq t\},$$
$$\nu_{j,t}:=\#\{d|\alpha_d\leq t\}+\#\{e|b_e\geq b_j,\beta_e\leq t\}.$$

The Schubert variety $\Sigma_{a^\alpha;b^\beta}$ is then defined to be the closure of the following locus:
$$\Sigma^\circ_{a^\alpha;b^\beta}:=\{(\Lambda_1,...,\Lambda_k)\in OF|\dim(F_{a_i}\cap \Lambda_t)= \mu_{i,t},\dim(F^\perp_{b_j}\cap \Lambda_t)= \nu_{j,t}\}.$$
We ask a similar question as Problem \ref{problem}:
\begin{Prob}\label{problem2}
Let $X$ be a subvariety in $OF(d_1,...,d_k;n)$ with class $[X]=\sigma_{a^\alpha;b^\beta}$. For each $1\leq i\leq s$, does there exist an isotropic subspace $F_{a_i}$ of dimension $a_i$ such that 
$$\dim(F_{a_i}\cap \Lambda_t)\geq \mu_{i,t}, \text{ for all } (\Lambda_1,...,\Lambda_k)\in X\text{ ? }$$
For each $1\leq j\leq d_k-s$, does there exist an isotropic subspace $F_{b_j}$ of dimension $b_j$ such that 
$$\dim(F^\perp_{b_j}\cap \Lambda_t)\geq \nu_{j,t},\text{ for all } (\Lambda_1,...,\Lambda_k)\in X \text{ ? }$$
\end{Prob}
If for some $i$ or $j$, the answer is true for all the representatives of $\sigma_{a^\alpha;b^\beta}$, then we call the corresponding $a_i$ or $b_j$ {\em rigid}. One may expect as in the previous section that, an essential sub-index $a_i$ or $b_j$ is rigid if and only if it is rigid with respect to the Schubert class $(\pi_t)_*(\sigma_{a^\alpha;b^\beta})$ for some $t$. However, it turns out that this condition is sufficient but not necessary. It can be seen from the following observations in orthogonal Grassmannians:
\begin{itemize}
\item If $a_i=b_j$ for some $i$ and $j$, then one should expect the same rigidity on both $a_i$ and $b_j$;
\item If a sub-quadric $Q_{n-b_j}\subset Q$ contains another sub-quadric of the maximal possible corank in $Q$, then the quadric $Q_{n-b_j}$ should also have the maximal possible corank. Therefore if $b_{j_1}>b_{j_2}$ and $b_{j_2}$ is essential, then the rigidity of $b_{j_1}$ should imply the rigidity of $b_{j_2}$.
\end{itemize}
To comply these observations, we introduce a new relation between essential sub-indices: 

\begin{Def}
Let $\sigma_{a^\alpha;b^\beta}$ be a Schubert class in $OF(d_1,...,d_k;n)$. Set
 $$E:=\{a_i|a_i \text{ essential}\}\cup\{b_j|b_j\text{ essential}\}.$$ We define a relation `$\Rightarrow$' between two sub-indices: $b_{j_1}\Rightarrow b_{j_2}$ if $j_2<j_1\neq\frac{n}{2}-1$ and $b_{j_2}$ is essential in $(\pi_t)_*(\sigma_{a^\alpha;b^\beta})$ for some $t\geq \min(\beta_{j_1},\beta_{j_2})$; $a_i\Rightarrow b_j$ and also $b_j\Rightarrow a_i$ if $a_i=b_j\neq \frac{n}{2}-1$. We then extend this relation to $E$ by transitivity and reflexivity (which we also denote by `$\Rightarrow$').
\end{Def}

Our next main result is the answer to Problem \ref{problem2}:
\begin{Thm}
Let $\sigma_{a^\alpha;b^\beta}$ be a Schubert class in $OF(d_1,...,d_k;n)$. An essential sub-index $a_i$ or $b_j$ is rigid if and only if there exists an element $e\in E$ such that $e\Rightarrow a_i$ (or $e\Rightarrow b_j$ resp.) and the sub-index $e$ is rigid with respect to the class $(\pi_t)_*(\sigma_{a^\alpha;b^\beta})$ for some $t$.
\end{Thm}

As an application, we obtain the classification of rigid Schubert classes in orthogonal partial flag varieties. Similar to the previous section, one can expect a Schubert class to be rigid if all essential sub-indices are rigid and compatible.
We introduce a similar relation on the set of essential sub-indices that reflects the compatibility:
\begin{Def}
Let $\sigma_{a^\alpha;b^\beta}$ be a Schubert class in $OF(d_1,...,d_k;n)$. Set
 $$E:=\{a_i|a_i \text{ essential}\}\cup\{b_j|b_j\text{ essential}\}.$$ We define a relation `$\rightarrow$' between two sub-indices: 
\begin{itemize}
\item $a_{i_1}\rightarrow a_{i_2}$ if $i_1\leq i_2$ and $a_{i_2}$ is essential in $(\pi_t)_*(\sigma_{a^\alpha;b^\beta})$ for some $t\geq \min(\alpha_{i_1},\alpha_{i_2})$. (If $n$ is even and $a_{i_2}=\frac{n}{2}$, then we do not require $a_{i_2}$ to be essential.) ; 
\item $a_i\rightarrow b_j$ if $a_i\leq b_{j}$;
\item $b_j\rightarrow a_i$ if $b_j\leq a_i$, $b_j\neq \frac{n}{2}-1$ and both $a_i$ and $b_j$ are essential in $(\pi_t)_*(\sigma_{a^\alpha;b^\beta})$ for some $t\geq \min(\alpha_i,\beta_{j})$; 
\item $b_{j_1}\rightarrow b_{j_2}$ if $j_1\leq j_2$ and $b_{j_1}$ is essential in $(\pi_t)_*(\sigma_{a^\alpha;b^\beta})$ for some $t\geq \min(\beta_{j_1},\beta_{j_2})$. 
\end{itemize}
We then extend this relation to $E$ by transitivity (which we also denote by `$\rightarrow$').
\end{Def}

\begin{Thm}
Let $\sigma_{a^\alpha;b^\beta}$ be a Schubert class in $OF(d_1,...,d_k;n)$. The class $\sigma_{a^\alpha;b^\beta}$ is rigid if and only if all essential indices are rigid and the set $E$ is totally ordered under the relation `$\rightarrow$'. 
\end{Thm}

\subsection{Generalized rigidity problem in Grassmannians}Let $X$ be an irreducible subvariety in the Grassmannian $G(k,n)$ with $[X]=\sum c_a\sigma_{a}$, where $\sigma_a$ are Schubert classes and $c_a$ are non-negative integers. For each $1\leq i\leq k$, set $$\gamma_i(X):=\min\{d|\exists F_d\in G(d,n)\text{ s.t. }\dim(F_d\cap \Lambda)\geq i, \forall \Lambda\in X\}.$$
We can ask the following question:
\begin{Prob}
Let $X$ be an irreducible subvariety with class $[X]=\sum c_a\sigma_{a}$. For each $1\leq i\leq k$, when does $\gamma_i(X)$ only depend on the class $[X]$?
\end{Prob}

In \S\ref{multi}, we partially answer this problem:
\begin{Thm}
Let $X$ be an irreducible subvariety with class $[X]=\sum_{a\in A}c_a\sigma_a$, $c_a\in \mathbb{Z}^+$. Set $m_i=\max\limits_{a\in A}\{a_i\}$. Define $$A_1:=\{a\in A|a_1=m_1\}$$ and inductively
$$A_i:=\{a\in A_{i-1}|a_i=\max\limits_{a'\in A_{i-1}}\{a_i'\}\},\  2\leq i\leq k.$$
If $\max\limits_{a\in A_i}\{a_i\}=m_i$ and there exists an index $a\in A_i$ such that $a_{i-1}+1=m_i\leq a_{i+1}-3$, then $\gamma_i(X)=m_i$.
\end{Thm}
A Schubert class is called {\em multi-rigid} or {\em Schur rigid} if every multiple of it can only be represented by a union of Schubert varieties. As an application of the previous result, we prove the multi-rigidity of certain Schubert classes in orthogonal Grassmanians:
\begin{Cor}
Let $\sigma_{a;b}$ be a Schubert class in $OG(k,n)$. Let $x_j:=\{i|a_i\leq b_j\}$. Set 
$a_0:=0$ and $$a_{s+1}:=n-b_{k-s}-(s-x_j)-1.$$ If for all $1\leq j\leq k-s$, $$x_j\geq k-j+1-\left[\frac{n-b_j-(a_{x_j+1}-1)}{2}\right]$$ and all the essential sub-indices $a_i$ satisfy the condition $$a_{i-1}+1=a_i\leq a_{i+1}-3$$ and for every essential sub-index $b_j$, $b_j=a_i$ for some $i$, then $\sigma_{a;b}$ is multi-rigid.
\end{Cor}

\subsection{Applications to symplectic Grassmannians}
The previous results can also be applied to symplectic Grassmannians $SG(k,n)$ (of type $C$).
\begin{Cor}
The Schubert class $\sigma_{a;b}$, where $a=(1,...,i)$ and $b=(i,i+1,...,k-1)$, $1\leq i\leq k$ is rigid for $SG(k,n)$, $n\geq 2k+2$.
\end{Cor}

Similar to the previous sections, the rigidity in symplectic partial flag varieties $SF(d_1,...,d_k;n)$ can be deduced from symplectic Grassmannians:
\begin{Prop}
Let $\sigma_{a^\alpha;b^\beta}$ be a Schubert class in $SF(d_1,...,d_k;n)$. Let $X$ be a representative of $\sigma_{a^\alpha;b^\beta}$. Let $F_{a_i}$ and $F_{b_j}$ be isotropic subspaces of dimension $a_i$ and $b_j$ respectively. If for some $\alpha_i\leq t\leq k$, $$\dim(F_{a_i}\cap\Lambda_t)\geq \mu_{i,t}\text{ for all }(\Lambda_1,...,\Lambda_k)\in X,$$ then the same inequailty holds for all $\alpha_i\leq t'\leq k$.

Similarly, if for some $\beta_j\leq t\leq k$, $$\dim(F_{b_j}^\perp\cap\Lambda_t)\geq \nu_{j,t}\text{ for all }(\Lambda_1,...,\Lambda_k)\in X,$$ then the same inequality holds for all $\beta_j\leq t'\leq k$.
\end{Prop}

The rigidity problem arose from the smoothing problem which was asked by Borel and Haefliger in 1961 \cite{BH}. A Schubert class is called smoothable if it can be represented by smooth subvarieties. The smoothable Schubert classes in Grassmannians were partially classified by Coskun in 2011 \cite{Coskun2011RigidAN}. Since most Schubert varieties are singular, and for a singular Schubert class, it can not be smoothable if it is rigid, our results in particular give a large class of non-smoothable Schubert classes in partial flag varieties.

The multi-rigidity problem was well studied by differential geometers. Walters and Bryant introduced the Schur differential systems and used it to prove the rigidity of certain Schubert classes \cite{Walter},\cite{RB2000}. Based on their works, Hong proved the rigidity of smooth Schubert classes in irreducible compact Hermitian symmetric space \cite{Ho1}, and a fully classification of multi-rigid classes in such spaces was obtained by Robles and The \cite{RT}. Since a multi-rigid class must be rigid, our results show the non-multi-rigidity of certain Schubert classes in partial flag varieties.

Most recently, Hong and Mok proved the rigidity \cite{HM} and multi-rigidity \cite{HM2} of smooth Schubert varieties in rational homogeneous spaces of Picard number one, using the geometric theory of uniruled projective manifolds. Different from their method, our approach works for the Schubert classes that are not necessarily smooth.

\subsection*{Organization of the paper} 
In \S\ref{sec-prelim}, we review the basic definitions and introduce our notation in partial flag varieties. In \S\ref{type A}, we study the rigidity problem in partial flag varieties of type A. In particular, we classify the rigid Schubert classes in type A partial flag varieties. In \S\ref{type BD}, we investigate the case of orthogonal partial flag varieties, which are of type B or D. In \S\ref{multi}, we generalize the rigidity problem to rational equivalence classes which are not necessary Schubert classes. We also proved the multi-rigidity of certain Schubert classes in orthogonal Grassmannians. In \S\ref{type C}, we apply the previous results to the symplectic cases (of type C).
\subsection*{Acknowledgments} 
We would like to thank Izzet Coskun, Yuri Tschinkel and others for valuable comments and corrections to improve our manuscript.

\section{Preliminaries}\label{sec-prelim}
In this section, we review basic facts and introduce our notation for flag varieties. Most of them are stated without proof. The proofs can be found in many textbooks in algebraic geometry, e.g. \cite{3264}. Fix once and for all a complex vector space $V$ of dimension $n$.
\subsection{Grassmannian varieties}\label{Grassmannian}
The Grassmannian variety $G(k,V)=G(k,n)$ parametrizes all $k$-dimensional vector subspaces in $V$. It is smooth and irreducible of dimension $k(n-k)$. A Schubert index $a_\bullet$ for $G(k,n)$ is an increasing sequence of positive integers
$$1\leq a_1<...<a_k\leq n.$$
Given a complete flag $F_\bullet:F_1\subset...\subset F_n=V$ of subspaces of $V$ and a Schubert index $a_\bullet$, the Schubert variety $\Sigma_a(F_\bullet)$ is defined as the following locus:
$$\Sigma_a(F_\bullet):=\{\Lambda\in G(k,n)|\dim(\Lambda\cap F_{a_i})\geq i\}.$$
It is an irreducible subvariety of dimension $\sum_{i=1}^k(a_i-i)$. Its rational equivalence class is independent of the choice of $F_\bullet$ and is denoted by $\sigma_a\in A^*(G(k,n))$. The Schubert classes form a free basis of $A^*(G(k,n))$. The Poincar\'e dual of $\sigma_a$ is the Schubert class $\sigma_{a^*}$, where $a^*=(n-a_k+1,...,n-a_1+1)$.

\begin{Prop}
Let $\sigma_{a}$ and $\sigma_b$ be two Schubert classes in $G(k,n)$. Then $\sigma_{a}\cdot\sigma_b=0$ if and only if $$a_i+b_{k-i+1}\leq n,\ \ \text{for some }1\leq i\leq k.$$
\end{Prop}

\begin{Rem}
In literature, it is customary to index a Schubert variety using an non-increasing sequence $\lambda_\bullet$ which reflects the codimension of it. Our notation can be translated into their notation using the relation
$$\lambda_i=n-k+i-a_i.$$
One of the advantages of our notation is that it is independent of the dimension of the ambient space $V$. Our notation can also be easily adapted to orthogonal Grassmannians as well as partial flag varieties.
\end{Rem}

\subsection{Partial flag varieties}
The partial flag variety $F=F(d_1,...,d_k;n)$ is defined to be the following locus: 
$$F(d_1,...,d_k;n):=\{(\Lambda_1,...,\Lambda_k)|\Lambda_i\subset\Lambda_{i+1},\Lambda_i\in G(d_i,V), 1\leq i\leq k\}.$$
(Here we set $\Lambda_{k+1}:=V$ and $d_{k+1}=n$). It is a smooth projective variety of dimension $\sum_{i=1}^k d_i(d_{i+1}-d_{i})$. 

The Schubert varieties in $F$ can be defined as follows: Take a Schubert index $$1\leq a_1<...< a_{d_k}\leq n$$ in the Grassmannian $G(d_k,n)$ and divide the sub-indices into $k$ blocks such that the $i$-th block has length $d_i-d_{i-1}$. We use the upper index to indicate which block the sub-index belongs to, and the resulting sequence
$a^\alpha=(a_1^{\alpha_1},..., a_{d_k}^{\alpha_{d_k}})$ is called a Schubert index in $F(d_1,...,d_k;n)$.

For each pair $(i,j)$, $1\leq i\leq d_k$, $1\leq j\leq k$, set 
$$\mu_{i,j}:=\#\{s|a_s\leq a_i,\alpha_s\leq j\}$$

\begin{Def}
Given a Schubert index $a^\alpha$ for $F(d_1,...,d_k;n)$ and a complete flag of vector subspaces $F_1\subset...\subset F_n=V$, where $\dim(F_i)=i$, the Schubert variety $\Sigma_{a^\alpha}(F_\bullet)$ is defined to be the following locus:
$$\Sigma_{a^\alpha}(F_\bullet):=\{(\Lambda_1,...,\Lambda_k)\in F(d_1,...,d_k;n)|\dim(F_{a_i}\cap \Lambda_j)\geq \mu_{i,j},1\leq i\leq d_k,1\leq j\leq k\}.$$
\end{Def}
The class of $\Sigma_{a^\alpha}(F_\bullet)$ is independent of the choice of $F_\bullet$ and is denoted by $\sigma_{a^\alpha}\in A^*(F)$. The Schubert classes $\sigma_{a^\alpha}$ form a free basis of $A^*(F)$. The Poincar\'e dual of $\sigma_{a^\alpha}$ is the Schubert class $\sigma_{b^{\beta}}$, where $b=a^*=(n-a_{d_k}+1,...,n-a_1+1)$, and $\beta=(\alpha_{d_k},\alpha_{d_k-1},...,\alpha_1)$.

\subsection{Orthogonal Grassmannians}\label{og}
Let $q$ be a non-degenerate symmetric bilinear form on $V$. A linear subspace $W$ is called isotropic with respect to $q$ if $q(W,W)=0$. The maximal dimension of isotropic subspaces is $\left[\frac{n}{2}\right]$, where $n=\dim(V)$. For $k<\frac{n}{2}$, the orthogonal Grassmannian $OG(k,V)=OG(k,n)$ is the subvariety of $G(k,n)$ that parametrizes all isotropic $k$-subspaces with respect to $q$. If $n$ is even and $k=\frac{n}{2}$, then the space of $k$-dimensional isotropic subspaces has two irreducible components, and we let $OG(k,2k)$ denote one of the components.

\begin{Def}
A Schubert index for $OG(k,n)$ consists of two increasing sequences of integers $$1\leq a_1<...< a_s\leq \frac{n}{2}$$ and $$0\leq b_1<...< b_{k-s}\leq \frac{n}{2}-1,$$ of length $s$ and $k-s$ respectively, $0\leq s\leq k$, such that $a_i\neq b_j+1$ for all $1\leq i\leq s,1\leq j\leq k-s$. If $n=2k$, then we further require $s$ and $k$ have the same parity, i.e.
$$ s \equiv k \pmod{2}.$$
\end{Def}
Given an isotropic subspace $W$, we denote $W^\perp$ its orthogonal complement with respect to $q$. Fix a complete flag of isotropic subspaces $F_\bullet=F_1\subset...\subset F_{\left[n/2\right]}$. When $n$ is even, the orthogonal complement of $F_{n/2-1}$ is a union of two maximal isotropic subspaces, one is $F_{n/2}$ and the other one belongs to the different irreducible components than $F_{n/2}$. By abuse of notation, we denote also by $F_{n/2-1}^\perp$ the maximal isotropic subspace in the orthogonal complement of $F_{n/2-1}$ other than $F_{n/2}$. 
\begin{Def}
Given a Schubert index $(a_\bullet;b_\bullet)$ for $OG(k,n)$, the corresponding Schubert variety $\Sigma_{a;b}$ is defined to be the Zariski closure of the following locus:
$$\Sigma^\circ_{a;b}(F_\bullet):=\{\Lambda\in OG(k,n)|\dim(\Lambda\cap F_{a_i})= i, \dim(\Lambda\cap F_{b_j}^\perp)= k-j+1, 1\leq i\leq s, 1\leq j\leq k-s\}.$$
\end{Def}
The Schubert classes $[\Sigma_{a;b}(F_\bullet)]\in A(OG(k,n))$ generate the Chow ring $A(OG(k,n))$ and are independent of the choice of flags $F_\bullet$. We will omit $F_\bullet$ and denote them by $\sigma_{a;b}=[\Sigma_{a;b}(F_\bullet)]$.

\begin{Rem}\label{difference1}
If $a_i=b_j+1$ for some $i,j$, then for every $\Lambda\in\Sigma_{a;b}$, either $\Lambda\cap F_{a_i}=\Lambda\cap F_{b_j}$ or $\Lambda\cap F_{a_i}^\perp=\Lambda\cap F_{b_j}^\perp$. If one of the possibilities results in an empty set, then the other one can better reflect the geometry of the locus. If both are effective, then the resulting locus is reducible.
\end{Rem}

\subsection{Restriction varieties}
There is another important class of subvarieties in orthogonal Grassmannians: the {\em restriction varieties} which generalize the Schubert varieties. We will use restriction varieties later to construct counter-examples of non-rigid Schubert classes. In the rest of this subsection, we review some basic facts about the restriction varieties.

The bilinear form $q$ geometrically defines a quadric hypersurface $Q$ in $\mathbb{P}(V)$. Therefore in the definition of Schubert varieties, one may replace $F_{b_j}^\perp$ with the intersection $\mathbb{P}(F_{b_j}^\perp)\cap Q$ , which is a sub-quadric of $Q$ of corank $b_j$. The restriction varieties are then obtained by varying the coranks of the defining sub-quadrics. 

More precisely, let $Q_d^r$ denote a subquadric of $Q$ of corank $r$, which is obtained by restricting $Q$ to a $d$-dimensional linear subspace. 
\begin{Def}
\label{sequence}
Given a sequence consisting of isotropic subspaces $F_{a_i}$ of $V$ and subquadrics $Q_{d_j}^{r_j}$ :
$$F_{a_1}\subsetneq...\subsetneq F_{a_s}\subsetneq Q_{d_{k-s}}^{r_{k-s}}\subsetneq...\subsetneq Q^{r_1}_{d_1}$$ such that

$(1)$ For every $1\leq j\leq k-s-1$, the singular locus of $Q_{d_j}^{r_j}$ is contained in the singular locus of $Q_{d_{j+1}}^{r_{j+1}}$;

$(2)$ For every pair $(F_{a_i},Q_{d_j}^{r_j})$, $\dim(F_{a_i}\cap \text{Sing}(Q_{d_j}^{r_j}))=\min\{a_i,r_j\}$;

$(3)$ Either $r_i=r_1=n_{r_1}$ or $r_t-r_i\geq t-i-1$ for every $t>i$. Moreover, if $r_t=r_{t-1}>r_1$ for some $t$, then $d_i-d_{i+1}=r_{i+1}-r_i$ for every $i\geq t$ and $d_{t-1}-d_t=1$;

$(A1)$ $r_{k-s}\leq d_{k-s}-3$;

$(A2)$ $a_i-r_j\neq 1$ for all $1\leq i\leq s$ and $1\leq j\leq k-s$;

$(A3)$ Let $x_j=\#\{i|a_i\leq r_j\}$. For every $1\leq j\leq k-s$, $$x_j\geq k-j+1-\left[\frac{d_j-r_j}{2}\right].$$
The associated restriction variety is defined to be the Zariski closure of the following locus:
$$\Gamma^\circ(F_\bullet,Q_\bullet):=\{\Lambda\in OG(k,n)|\dim(\Lambda\cap F_{a_i})= i,\dim(\Lambda\cap Q_{d_j}^{r_j})= k-j+1,1\leq i\leq s, 1\leq j\leq k-s\}.$$
\end{Def}

Since the Schubert classes form a free basis of the Chow ring, it is then natural to ask how to compute the classes of restriction varieties in terms of Schubert classes. In \cite{Coskun2011RestrictionVA}, Coskun obtained a positive algorithm to compute the class of a restriction variety. The idea is to increase the corank of each sub-quadric one by one, until they all reach the maximal possible corank. If at some stages the corank becomes 1 less than some $a_i$, then it would split into two pieces due to the observation in Remark \ref{difference1}.

\begin{Ex}
Consider the restriction variety $\Gamma$ defined by $F_2\subset Q_{n-1}^0$ in $OG(2,n)$, $n\geq 7$. To compute the class of $\Gamma$, we specialize the sub-quadric $Q_{n-1}^0$ to a sub-quadric $Q_{n-1}^1$ of the same dimension, but with the corank $1$ greater. The sequence now becomes $F_2\subset Q_{n-1}^1$. Let $\mathbb{P}(F_1)$ be the singular locus of $Q_{n-1}^1$. By Remark \ref{difference1}, the sequence then splits in to two parts: $F_1\subset F_1^\perp$ and $F_2\subset F_2^\perp$, which correspond to the Schubert classes $\sigma_{1;1}$ and $\sigma_{2;2}$. Therefore $[\Gamma]=\sigma_{1;1}+\sigma_{2;2}$.
\end{Ex}

\subsection{Orthogonal partial flag varieties}\label{of}

Let $q$ be a non-degenerate symmetric bilinear form on $V$. Let $1\leq d_1<...<d_k\leq \frac{n}{2}$ be an increasing sequence of positive integers. For $d_k<\frac{n}{2}$, the orthogonal partial flag variety $OF(d_1,...,d_k;n)$ parametrizes all $k$-step partial flags $(\Lambda_1,...,\Lambda_k)$, where $\Lambda_i\subset \Lambda_{i+1}$ and $\Lambda_i$ are isotropic subspaces of dimension $d_i$. When $d_k=\frac{n}{2}$, then the space of isotropic partial flags has two irreducible components and we let $OF(d_1,...,d_k;n)$ denote one of the components.

\begin{Def}
A Schubert index for $OF(d_1,...,d_k;n)$, $n\geq 2d_k$, consists of two increasing sequences of integers $$1\leq a_1^{\alpha_1}<...<a^{\alpha_s}_s\leq \frac{n}{2}$$ and $$0\leq b^{\beta_1}_1<...<b^{\beta_{k-s}}_{k-s}\leq \frac{n}{2}-1,$$ of length $s$ and $k-s$ respectively, $0\leq s\leq k$, such that 
\begin{enumerate}
\item $a_i\neq b_j+1$ for all $1\leq i\leq s,1\leq j\leq k-s$, and 
\item $d_t-d_{t-1}=\#\{i|\alpha_i=t\}+\#\{j|\beta_j=t\}$ for each $1\leq t\leq k$. (Here we set $d_0=0$.)
\end{enumerate}
If $n=2d_k$, then we further require $s$ and $d_k$ have the same parity, i.e.
$$ s \equiv d_k \pmod{2}.$$
\end{Def}

\begin{Def}
Given a Schubert index $(a^\alpha,b^\beta)$ for $OF(d_1,...,d_k;n)$ and a flag of isotropic subspaces $F_1\subset ...\subset F_{\left[n/2\right]}$ in $V$, set
$$\mu_{i,t}:=\#\{c|a_c\leq a_i,\alpha_c\leq t\},$$
$$\nu_{j,t}:=\#\{d|\alpha_d\leq t\}+\#\{e|b_e\geq b_j,\beta_e\leq t\}.$$
The Schubert variety $\Sigma_{a^\alpha;b^\beta}$ is then defined to be the Zariski closure of the following locus:
$$\Sigma_{a^\alpha;b^\beta}:=\{(\Lambda_1,...,\Lambda_k)\in OF(d_1,...,d_k;n)|\dim(F_{a_i}\cap \Lambda_t)= \mu_{i,t},\dim(F^\perp_{b_j}\cap \Lambda_t)= \nu_{j,t}\}.$$
\end{Def}

\section{The rigidity problem in partial flag varieties}\label{type A}
In this section, we study the rigidity problem in partial flag varieties of type A. 
\begin{Def}
Given a Schubert class $\sigma_a$ in $A^*(F(d_1,...,d_k;n))$, $\sigma_a$ is called {\em rigid} if it can only be represented by Schubert varieties.
\end{Def}
\begin{Ex}
Consider the Grassmannian variety $G(k,n)=F(k;n)$. Let $F_\bullet$ be a complete flag of subspaces in $V$. The Schubert variety $\Sigma_{n-k,n-k+2,...,n}(F_\bullet)$ is a singular hyperplane section of $G(k,n)$ under the Plucker embedding, while a general hyperplane section of $G(k,n)$ is smooth by Bertini's theorem. Therefore the Schubert class $\sigma_{n-k,n-k+2,...,n}$ is not rigid.
\end{Ex}

We first review the rigidity results in Grassmannians, which are the special cases of partial flag varieties when $k=1$, then we generalize them to general partial flag varieties of type A.

\subsection{Rigidity results in Grassmannians}

The main idea is to investigate the rigidity of each sub-index. We first restrict our attention to the "essential" ones.
\begin{Def}
Let $\sigma_a$ be a Schubert class in $G(k,n)$. A sub-index $a_i$ is called {\em essential} if either $i=k$ or $a_{i}<a_{i+1}-1$.
\end{Def}
\begin{Prop}\label{essential a}
Let $X$ be a representative of $\sigma_a$ in $G(k,n)$. If there exists a flag $F_\bullet$ of subspaces such that for all $\Lambda\in X$, $\dim(\Lambda\cap F_{a_i})\geq i$ for all essential $a_i$, then $X$ has to be the Schubert variety defined by $F_\bullet$.
\end{Prop}
\begin{proof}
It follows from the definition of Schubert varieties and the observation that $\dim(\Lambda\cap F_{a_{i+1}})\geq i+1$ implies $\dim(\Lambda\cap F_{a_i})\geq i$ if $F_{a_i}$ is of codimension 1 in $F_{a_{i+1}}$.
\end{proof}

This lead to the following definition of rigid sub-indices:
\begin{Def}
Let $\sigma_a$ be a Schubert class in $G(k,n)$. An essential sub-index $a_i$ is called {\em rigid} if for every representative $X$ of $\sigma_a$, there exists a vector subspace $F_{a_i}$ of dimension $a_i$ such that
$$\dim(\Lambda\cap F_{a_i})\geq i,\forall \Lambda\in X.$$
\end{Def}
In \cite[Theorem 1.3]{YL}, we classified the rigid sub-indices in Grassmannians:
\begin{Thm}\cite[Theorem 1.3]{YL}\label{rigid index in g}
Let $\sigma_a$ be a Schubert class in $G(k,n)$. An essential sub-index $a_i$ is rigid if and only if one of the following conditions holds:
\begin{enumerate}
\item $i=k$; or

\item $a_i=i$; or

\item $a_i\leq a_{i+1}-3$; or

\item $a_i=a_{i-1}+1$.
\end{enumerate}
\end{Thm}
In particular, we obtained the classification of rigid Schubert varieties in Grassmannians:
\begin{Thm}\cite[Theorem 3.5]{YL}
Let $\sigma_a$ be a Schubert class in $G(k,n)$. Then $\sigma_a$ is rigid if and only if all essential sub-indices are rigid.
\end{Thm}

\begin{Rem}\label{ex of non rigid}
Let $\sigma_a$ be a Schubert class in $G(k,n)$. If for some $1\leq i\leq k$,
$$a_{i-1}+2\leq a_i=a_{i+1}-2,$$ then we can construct a non-Schubert variety that represents the Schubert class $\sigma_a$ in $G(k,n)$. The construction is due to Coskun in \cite[Theorem 1.3]{Coskun2011RigidAN}.

Let $a_j$ be the first essential sub-index after $a_i$, i.e. the largest $j$ such that 
$$a_j=a_i+j-i+1.$$
Take an $a_j$-dimensional vector subspace $F_{a_j}$. Consider the Schubert class $\sigma_\mu$, $\mu=(\mu_1,...,\mu_\alpha)$, which corresponds to the Schubert class $\sigma_{a_1,...,a_j}$ under the isomorphism
$$G(j,F_{a_j})\cong G(a_j-j,F_{a_j}),$$
which takes a $j$-plane to its dual in $F_{a_j}$. Then the condition $a_{i-1}+2\leq a_i=a_{i+1}-2$ translates to the conditions $\mu_1>1$ and $\mu_2=\mu_1+2$.

Let $\mu_t$ be the first essential sub-index after $\mu_1$, i.e. the largest $t$ such that $$\mu_t=\mu_1+t.$$ 
Pick a partial flag $W_{\mu_t}\subset...\subset W_{\mu_\alpha}$ of subspaces of $F_{\lambda_j}$. Let $H$ be a smooth hyperplane section of $G(t,W_{\mu_t})$ under the Pl\"ucker embedding. Consider the following variety
$$Y:=\{L\in G(a_j-j,F_{a_j})|L'\subset L\cap W_{\mu_t},\text{ for }L'\in H,\text{ and }\dim(L\cap W_{\mu_s})\geq s,t\leq s\leq \alpha\}.$$
Then $Y$ represents the Schubert class $\sigma_{\mu}$, but is not a Schubert variety since $H$ is not. 

Pick a partial flag $F_{a_j}\subset...\subset F_{a_k}$. Let $X$ be the variety defined as follows:
$$X:=\{\Lambda\in G(k,n)|\Lambda'\subset \Lambda\cap F_{a_j},\text{ for }\Lambda'\in Y^*, \text{ and }\dim(\Lambda\cap F_{a_s})\geq s, j\leq s\leq k\}.$$
Then $X$ is a Schubert variety representing the Schubert class $\sigma_a$, but is not a Schubert variety.
\end{Rem}

\subsection{Classes arose from the canonical projections}
For partial flag varieties, we will study the rigidity problem by tracing the essential sub-indices to sub-indices in Grassmannians via canonical projections 
$$\pi_t:F(d_1,...,d_k;n)\rightarrow G(d_t,n).$$
In this subsection, we calculate the push-forward of Schubert classes under the canonical projections and calculate the class of general fibers of the restriction of $\pi_t$ to a Schubert variety.

\begin{Prop}\label{pushforward of class}
Let $\pi_t$ be the $t$-th canonical projection $F(d_1,...,d_k;n)\rightarrow G(d_t,n)$. Let $\Sigma_{a^\alpha}$ be a Schubert variety in $F(d_1,...,d_k;n)$ and $X$ be a subvariety which is rationally equivalent to $\Sigma_{a^\alpha}$. Then $\pi_t(X)$ is rationally equivalent to $\pi_t(\Sigma_{a^\alpha})$.
\end{Prop}
\begin{proof}
It is easy to see that $\pi_t(\Sigma_{a^\alpha})$ is the Schubert variety $\Sigma_b$ in $G(d_i,n)$, where $b$ is an increasing sequence obtained from $a$ by removing all the sub-indices $a_i$ whenever $\alpha_i$ is greater than $t$.

Now it suffices to show that $[\pi_t(X)]=\sigma_b$. We use the method of undetermined coefficients. Let $b^*$ be the dual index $(n+1-b_{d_i},...,n+1-b_1)$ of $b$. Let $c$ be a Schubert index such that $\dim(\sigma_c)=\dim(\sigma_{b^*})$. Then $\sigma_b\cdot\sigma_c=\sigma_{1,...,d_i}$ is the point class if $c=b^*$, and $\sigma_b\cdot \sigma_c=0$ if $c\neq b^*$. Let $\Sigma_c$ be a Schubert variety in $G(d_t,n)$ defined by a general flag in $V$. If $c\neq b^*$, then $\sigma_b\cdot\sigma_c=0$ implies $\sigma_{a^\alpha}\cdot [\pi^{-1}(\Sigma_c)]=0$ and therefore $X\cap\pi^{-1}_t(\Sigma_c)=\emptyset$. Thus $[\pi_t(X)]\cdot\sigma_c=0$ for $c\neq b^*$. Similarly, consider the Schubert variety $\Sigma_{{b'}^\beta}$ in $F$ defined by a general flag, where $b'_i=n+1-a_i$ and $\beta_i=\alpha_i$. The Schubert class $\sigma_{{b'}^\beta}$ is the dual class of $\sigma_{a^\alpha}$, and in particular $\sigma_{{b'}^\beta}\cdot \sigma_{a^\alpha}$ is the point class in $A^*(F)$. This implies $X\cap \Sigma_{{b'}^\beta}$ consists of one point in $F$, and hence $\pi_t(X\cap\Sigma_{{b'}^\beta})$ also consists of one point in $G(d_t,n)$. Notice that $[\pi_t(\Sigma_{{b'}^\beta})]=\sigma_{b^*}$ and therefore $[\pi_t(X)]\cdot\sigma_{b^*}$ is the point class in $A^*(G(d_t,n))$. This shows that $[\pi_t(X)]=\sigma_b$.
\end{proof}

Let $X$ be a subvariety in $F(d_1,...,d_k;n)$ representing the Schubert class $\sigma_{a^\alpha}$. Let $X_i:=\pi_i(X)$ be its image under the $i$-th projection. Let $\pi_i|_{X}:X\rightarrow X_i$ be the restriction of $\pi_i$ to $X$.
\begin{Prop}\label{fiberk}
A general fiber of $\pi_k|_X$ has class $\sigma_{b^\beta}$ where $b^\beta=(1^{\alpha_1},...,d_k^{\alpha_{d_k}})$.
\end{Prop}
\begin{proof}
First assume $X$ is the Schubert variety, which is defined by the partial flag
$$F_{a_1}^{\alpha_1}\subset...\subset F_{a_{d_k}}^{\alpha_{d_k}}.$$
Let $\Lambda_k$ be a general point in $X_k$. Observe that $\Lambda_k\cap F_{a_i}$ impose a corank condition on the $j$-flag element $\Lambda_j$ if and only if $\alpha_i\leq j$. The fiber over $\Lambda_k$ is the Schubert variety defined by
$$(\Lambda_k\cap F_{a_1})^{\alpha_1}\subset...\subset (\Lambda_k\cap F_{a_{d_k}})^{\alpha_{d_k}}.$$
Since for a general $\Lambda_k\in X_k$, $\dim(\Lambda_k\cap F_{a_i})=i$, the fiber has class $\sigma_{b^\beta}$, where $b^\beta=(1^{\alpha_1},...,d_k^{\alpha_{d_k}}$).

Now consider a general representative $X$ which is not necessarily a Schubert variety. Let $a^*=(n-a_1+1,...,n-a_{d_k}+1)$ be the dual Schubert index of $a_\bullet$ in $G(d_k,n)$. Then for a general point $\Lambda_k$ in $X_k$, there is a Schubert variety $\Sigma_{a^*}$ in $G(d_k,n)$ such that $X_k$ and $\Sigma_{a^*}$ intersect transversally at $\Lambda_k$. Therefore the class of the fiber of $\pi_k|_X$ at $\Lambda_k$ is given by the intersection product $\sigma_{a^\alpha}\cdot\pi_k^*(\sigma_{a^*})$, which equals $\sigma_{b^\beta}$ by replacing $X$ with a Schubert variety and applying the previous argument.

\end{proof}
Now we consider the case when $i=1$. Let $a'^{\alpha'}$ be the sequence obtained from $a^\alpha$ by removing the elements whose upper index equals $1$.
\begin{Prop}\label{fiber1}
A general fiber of $\pi_1|_X$ has class $\sigma_{b^\beta}$ where $b_i=i$ and $\beta_i=1$ for $1\leq i\leq d_1$; $b_{d_1+i}=a_i'+\#\{s|a_s>\alpha_i',\alpha_s=1\}$ and $\beta_{d_1+i}=\alpha_i'$ for $d_1+1\leq d_1+i\leq d_k$.
\end{Prop}
\begin{proof}
Let $a^1$ be the sub-sequence of $a^\alpha$ which consists of all elements whose upper index equals $1$. Let $(a^1)^*$ be the dual Schubert index of $a^1$ in $G(d_1,n)$. Then the class of a general fiber of $\pi_1|_X$ is given by the intersection product $\sigma_{a^\alpha}\cdot\pi_1^*(\sigma_{(a^1)^*})$. Therefore there is no harm to assume $X$ to be a Schubert variety. Assume $X$ is defined by
$$F_{a_1}^{\alpha_1}\subset...\subset F_{a_{d_k}}^{\alpha_{d_k}}.$$
Let $\Lambda_1$ be a general point in $X_1$. Then the fiber of $\pi_1|_X$ at $\Lambda_1$ is the Schubert variety defined by the flag
$$\Lambda_{1,d_1-1}^1\subset...\subset \Lambda^1_{1,1}\subset \Lambda^1_1\subset(\text{span}(\Lambda_1,F_{a_1'}))^{\alpha'_1}\subset...\subset (\text{span}(\Lambda_1,F_{a'_{d_k-d_1}}))^{\alpha_{d_k-d_1}'},$$
where $\Lambda_{1,1}$ is a general codimension 1 linear subspace of $\Lambda_1$, and $\Lambda_{1,j}$ is a general codimension 1 linear subspace of $\Lambda_{1,j-1}$, $2\leq j\leq d_1-1$. Notice that span$(\Lambda_1,F_{a_i'})$ has dimension $a_i'+\#\{s|a_s>\alpha_i',\alpha_s=1\}$, and the statement then follows.
\end{proof}

For $1<i<k$, let $a'^{\alpha'}$ be the sequence obtained from $a^\alpha$ be removing the elements whose upper index is greater than $i$, and let $a''^{\alpha''}$ be the sequence obtained from $a^\alpha$ by removing $a'^{\alpha'}$.
\begin{Cor}\label{class of fiber}
A general fiber of $\pi_i|_X$ has class $\sigma_{b^\beta}$ where $b_j=j$ and $\beta_j=\alpha_j$ for $1\leq j\leq d_i$; $b_{d_i+t}=a_t''+\#\{s|a_s>\alpha_t',\alpha_s\leq i\}$ and $\beta_{d_i+t}=\alpha_t'$ for $d_i+1\leq d_i+t\leq d_k$.
\end{Cor}

\begin{proof}
Consider the two-step projection 
$$F(d_1,...,d_k;n)\rightarrow F(d_1,...,d_i;n)\rightarrow F(d_i,n).$$
The fibers of the first projection are isomorphic to the fibers of the projection $$F(d_i,...,d_k;n)\rightarrow G(d_i,n).$$ The statement then follows by applying Proposition \ref{fiber1} to the projection $F(d_i,...,d_k;n)\rightarrow G(d_i,n)$ and applying Proposition \ref{fiberk} to the projection $F(d_1,...,d_i;n)\rightarrow G(d_i,n)$.
\end{proof}

\subsection{General partial flag varieties of type A}
In this section, we study the rigidity problem in general partial flag varieties of type A. The basic idea is to trace the rigidity of essential indices under the canonical projection from the flag varieties to Grassmannians. 

We start with an example:
\begin{Ex}
Consider the Schubert class $\sigma_{2^1,4^2}$ and $\sigma_{2^2,4^1}$ in the partial flag variety $F(1,2,4)$. Under the canonical projection $\pi_2:F(1,2;4)\rightarrow G(2,4)$, both classes are taken to the class $\sigma_{2,4}\in A(G(2,4))$ which is not a rigid Schubert class. We claim that $\sigma_{2^1,4^2}$ is rigid while $\sigma_{2^2,4^1}$ is not.

First consider the Schubert class $\sigma_{2^1,4^2}$. Let $X$ be a representative of it. Let $\pi_1$ be the first projection $\pi_1:F(1,2;4)\rightarrow G(1,4)$. Then $[\pi_1(X)]=\sigma_2\in A(G(1,4))$ which is a rigid Schubert class by Theorem \ref{rigid index in g}. Let $F_2$ be the unique vector subspace of dimension $2$ such that $\Lambda_1\subset F_2$ for all $\Lambda_1\in \pi_1(X)$. Then for every $\Lambda_2\supset \Lambda_1$, $\dim(\Lambda_2\cap F_2)\geq 1$. Therefore $X$ is contained in the following set
$$\{(\Lambda_1,\Lambda_2)|\Lambda_1\subset F_2,\dim(\Lambda_2\cap F_2)\geq 1\},$$
which is the Schubert variety $\Sigma_{2^1,4^2}(F_2\subset V)$. Since $\dim(X)=\dim(\Sigma_{2^1,4^2})$ and the Schubert varieties are irreducible, $X$ is the Schubert variety $\Sigma_{2^1,4^2}$.

Then consider the Schubert class $\sigma_{2^2,4^1}$. Let $H$ be a smooth hyperplane section of $G(2,4)$ under the Plucker embedding. Let $Y$ be the variety defined as follows:
$$Y:=\{(\Lambda_1,\Lambda_2)|\Lambda_2\in H,\Lambda_1\subset\Lambda_2\}.$$
Then $[Y]=\sigma_{2^2,4^1}$, but is not a Schbuert variety since $H$ is not.
\end{Ex}
\begin{Def}
Let $a^\alpha$ be a Schubert index in the partial flag variety $F=F(d_1,...,d_k;n)$. A sub-index $a_i$ is called {\em essential} if it is essential with respect to the class $(\pi_t)_*(\sigma_{a^\alpha})$ for some $1\leq t\leq k$.
\end{Def}
\begin{Prop}\label{essential flag}
Let $X$ be a representative of $\sigma_{a^\alpha}$ in $F(d_1,...,d_k;n)$. If there exists a flag $F_\bullet$ of subspaces such that for all $(\Lambda_1,...,\Lambda_k)\in X$, $\dim(F_{a_i}\cap \Lambda_j)\geq \mu_{i,j}$ for all essential $a_i$, then $X$ has to be the Schubert variety defined by $F_\bullet$.
\end{Prop}
\begin{proof}
Let $X_t:=\pi_t(X)$, $1\leq t\leq k$, be the image of $X$ under the $t$-th canonical projection. By assumption, for all sub-indices $a_i$ which are essential with respect to the class $(\pi_t)_*(\sigma_{a^\alpha})$, $\dim(F_{a_i}\cap \Lambda_t)\geq \mu_{i,t}$ for all $\Lambda_t\in X_t$. By Proposition \ref{essential a}, $X_t$ are Schubert varieties defined by $F_\bullet$ for all $1\leq t\leq k$, and therefore $X$ has to be the Schubert variety.
\end{proof}
Directly from the definition, we have the following characterization of essential sub-indices:
\begin{Prop}
Let $\sigma_{a^\alpha}$ be a Schubert class for $F(d_1,...,d_k;n)$. A sub-index $a_i$ is essential if and only if one of the following conditions holds:
\begin{itemize}
\item $i=d_k$;
\item $i<d_k$, $a_i<a_{i+1}-1$;
\item $i<d_k$, $\alpha_i<\alpha_{i+1}$.
\end{itemize}
\end{Prop}

\begin{Def}
An essential sub-index $a_i$ is called {\em rigid} if for every representative $X$ of $\sigma_{a^\alpha}$, there exists a linear space $F_{a_i}$ of dimension $a_i$ such that 
$$\dim(F_{a_i}\cap\Lambda_j)\geq \mu_{i,j}, \forall (\Lambda_1,...,\Lambda_k)\in X,\alpha_i\leq j\leq k.$$
\end{Def}

\begin{Thm}\label{rigid index in f}
Given a Schubert class $\sigma_{a^\alpha}$ in $F(d_1,...,d_k;n)$. If $a_i$ is rigid with respect to the Schubert class $(\pi_j)_*(\sigma_{a^\alpha})\in A(G(d_j,n))$ for some $j\geq \alpha_i$, then $a_i$ is rigid with respect to the Schubert class $\sigma_{a^\alpha}\in A(F(d_1,...,d_k;n))$.
\end{Thm}
\begin{proof}
Let $X$ be a representative of $\sigma_{a^\alpha}$. By assumption, there exists a linear space $F_{a_i}$ of dimension $a_i$ such that 
$$\dim(F_{a_i}\cap \Lambda_j)\geq \mu_{i,j},\ \forall \Lambda_j\in \pi_j(X).$$
We claim that the same inequality holds for all $\alpha_i\leq t\leq k$.

Suppose, for a contradiction, that $\dim(F_{a_i}\cap \Lambda_t)<\mu_{i,t}$ for some $\alpha_i\leq t\leq k$ and $\Lambda_t\in \pi_t(X)$. By the semi-continuity of dimension, this inequality holds for a general $\Lambda_t\in \pi_t(X)$. Consider the fiber of $\pi_t$ over a general point $\Lambda_t\in \pi_t(X)$. By Corollary \ref{class of fiber}, $\pi_t|_X$ has class $\sigma_{b^\beta}$ where $b_{\mu_{i,t}}=\mu_{i,t}$ and $\beta_{\mu_{i,t}}=\alpha_i$. 

If $j<t$, then by Proposition \ref{pushforward of class}, $({\pi_j}\circ \pi_i^{-1})(\Lambda_t)$ has class $\sigma_{c}$ where $c_{\mu_{i,j}}=\mu_{i,t}$. On the other hand, set $W=F_{a_i}\cap\Lambda_t$. By assumption, $w:=\dim(W)<\mu_{i,t}$. Since $\dim(F_{a_i}\cap\Lambda_j)\geq \mu_{i,j}$, we get for all $\Lambda'_j\in \pi_j(\pi_t^{-1}(\Lambda_t))$, 
$$\dim(W\cap\Lambda'_j)= \dim(F_{a_i}\cap\Lambda'_j)\geq \mu_{i,j}.$$
Since $\dim(W)<\mu_{i,t}$, it contradicts the condition that $c_{\mu_{i,j}}=\mu_{i,t}$.

If $t<j$, replacing $X$, if necessary, with its image under the projection $$F(d_1,...,d_k;n)\rightarrow F(d_1,...,d_j;n),$$ we may assume $j=k$ and $a_i$ is rigid (and hence essential) in the $k$-th component. If there is no $l$ such that $l>i$ and $\alpha_l>t$, then $d_k-d_t=\mu_{i,k}-\mu_{i,t}$. Since $\Lambda_t$ is of codimension $d_k-d_t$ in $\Lambda_k$, we must have
$$\dim(F_{a_i}\cap \Lambda_t)\geq \dim(F_{a_i}\cap\Lambda_k)-(d_k-d_t)\geq \mu_{i,k}-(\mu_{i,k}-\mu_{i,t})=\mu_{i,t}.$$
If there exists $l$ such that $l>i$ and $\alpha_l>t$, assume $l$ is minimum among such $l$. Set $\gamma=\#\{s|s\leq l,\alpha_s>t\}$. By Corollary \ref{class of fiber} and Proposition \ref{pushforward of class}, $({\pi_k}\circ \pi_i^{-1})(\Lambda_t)$ has class $\sigma_{c}$ where $$c_{d_t+\gamma}=d_t+\gamma+a_l-l.$$
On the other hand, let $U=$span$(\Lambda_t,F_{a_i})$. Then $\dim(U)=d_t+a_i-w$, and 
\begin{eqnarray}
\dim(U\cap \Lambda_k)=d_t+i-w&>&d_i+i-\mu_{i,t}\nonumber\\
&=&d_t+\gamma-1.\nonumber
\end{eqnarray}
Let $U'$ be a general subspace of codimension $i-w-\gamma$ in $U$. Then
$$\dim(U'\cap\Lambda_k)\geq d_t+i-w-(i-w-\gamma)=d_t+\gamma.$$
Since $a_i$ is essential, $a_l-l>a_i-i$, and therefore 
$$\dim(U')=d_t+a_i-w-(i-w-\gamma)<d_t+\gamma+a_l-l.$$
This contradicts the relation $c_{d_t+\gamma}=d_t+\gamma+a_l-l.$ We conclude that $a_i$ is rigid.

\end{proof}

The converse of Theorem \ref{rigid index in f} is also true. We can construct a counter-example when an essential $a_i$ is not rigid in all components. For simplicity, we assume $k=2$. The case when $k\geq 3$ can be constructed similarly.

Let $\sigma_{a^\alpha}$ be a Schubert class in $F(d_1,d_2;n)$. Assume a sub-index $a_i$ is essential but not rigid with respect to the classes $(\pi_j)_*(\sigma_{a^\alpha})$, $j=1,2$. There are two cases depending on whether $a_{i+1}=a_i+1$. For each case, we will first give an example and then provide a general construction.

Case I: $a_{i+1}\neq a_i+1$.
\begin{Ex}
Consider the Schubert class $\sigma_{1^1,3^2,5^2}$ in $F(2,3;5)$. Let $\pi:F(2,3;5)\rightarrow G(3,5)$ be the canonical projection. Then $\pi_*(\sigma_{1^1,3^2,5^2})=\sigma_{1,3,5}\in A(G(3,5)).$ By Theorem \ref{rigid index in g}, the sub-index $3$ is not rigid. The construction in Remark \ref{ex of non rigid} gives a non-Schubert subvariety $Y$ in $G(3,5)$ with class $[Y]=\sigma_{1,3,5}$, such that 
$$F_1\subset \Lambda\subset F_5,\text{ for all } \Lambda\in Y.$$
Let 
$$X:=\{(F_1,\Lambda_2)\in F(2,3;5)|\Lambda_2\in Y\}.$$
It is easy to check that $X$ has class $\sigma_{1^1,3^2,5^2}$, but is not a Schubert variety since $Y$ is not.
\end{Ex}
Construction of Case I: By assumption, we must have $a_{i-1}+2\leq a_i=a_{i+1}-2$. By Remark \ref{ex of non rigid}, there exists a non-Schubert subvariety $Y$ of $G(d_2,n)$ with class $[Y]=\sigma_a\in A(G(d_2,n))$ such that $$\dim(\Lambda\cap F_{a_j})\geq j,\ \text{ for all }\Lambda\in Y\text{ and }j\neq i,$$
for some $(d_2-1)$-step partial flag 
$$F_{a_1}\subset...\subset F_{a_{i-1}}\subset F_{a_{i+1}}\subset ...\subset F_{a_{d_2}}.$$

Define 
$$X:=\{(\Lambda_1,\Lambda_2)|\Lambda_2\in Y,\Lambda_1\subset \Lambda_2,\dim(F_{a_j}\cap\Lambda_1)\geq \mu_{j,1},\alpha_j=1,j\neq i\}$$
It is then straightforward to verify that $X$ is irreducible and has class $\sigma_{a^\alpha}$.

Let $\pi_s:F(d_1,d_2;n)\rightarrow G(d_s,n)$, $s=1,2$ be the canonical projections. Let $a'$ be the sub-sequence of $a$ obtained by removing all the sub-indices whose upper index is greater than $1$. The length of $a'$ equals $d_1$. The fiber of $\pi_2|_X$ at a point $\Lambda_2\in Y$ is isomorphic to the Schubert variety in $G(d_1,d_2)$ defined by the partial flag
$$F_{a'_1}\cap \Lambda_2\subset...\subset F_{a'_{d_1}}\cap \Lambda_2$$
(if $\alpha_i=1$ and $a'_{i'}=a_i$, then we omit the flag element $F_{a'_{i'}}\cap \Lambda_2$).
By the fiber dimension theorem, $X$ is irreducible. By specializing $Y$ to a Schubert variety, we get $\dim(X)=\dim(\Sigma_{a^\alpha})$.

To compute its class, we intersect $X$ with Schubert varieties of complementary dimension. The Poincar\'e dual of $\sigma_{a^\alpha}$ is the Schubert class $\sigma_{b^\beta}$, where $b=(n-a_{d_2}+1,...,n-a_1+1)$ and $\beta=(\alpha_{d_2},...,\alpha_1)$. Take a general Schubert variety $\Sigma_{b^\beta}$ defined by a partial flag $$G_{n-a_{d_2}+1}^{\alpha_{d_2}}\subset...\subset G_{n-a_1+1}^{\alpha_1}.$$
Notice that $\pi_2(\Sigma_{b^\beta})$ is the Schubert variety $\Sigma_{b}$ in $G(d_2,n)$. Since $\sigma_b$ is the Poincar\'e dual of $\sigma_a$ in $G(d_2,n)$, $Y$ intersect $\Sigma_{b}$ in a unique point $\Lambda_2$. By Corollary \ref{class of fiber}, $(\pi_2|_{\Sigma_{b^\beta}})^{-1}(\Lambda_2)$ has class $\sigma_{e^\beta}$, where $e=(1,2,...,d_2)$. By proposition \ref{pushforward of class} and the previous description of the fibers of $\pi_2|_X$, we can see that at the point $\Lambda_2$, the fiber of $\pi_2|_{\Sigma_{b^\beta}}$ has the dual class of the fiber of $\pi_2|_{X}$ in $G(d_1,\Lambda_2)$. Therefore they meet in a unique point $\Lambda_1\in G(d_1,\Lambda_2)$, and thus $(\Lambda_1,\Lambda_2)$ is the unique point contained in the intersection of $X$ and $\Sigma_{b^\beta}$.

On the other hand, if $\sigma_{c^\gamma}$ is a Schubert class of complementary dimension, $c^\gamma\neq b^\beta$, then we claim that $\sigma_{a^\alpha}\cdot \sigma_{c^\gamma}=0$. Let $\Sigma_{c^\gamma}$ be a general Schubert variety. First observe that if $\pi_2(\Sigma_{c^\gamma})\cap \pi_2(X)=\emptyset$, then $\Sigma_{c^\gamma}\cap X=\emptyset$. Since $[\pi_2(X)]=\sigma_a$ and $[\pi_2(\Sigma_{c^\gamma})=\sigma_c$, we get $\sigma_{a^\alpha}\cdot\sigma_{c^\gamma}=0$ if $a_j+c_{d_2-j+1}\leq n$ for some $j$. Now assume $c_{d_2-j+1}\geq n-a_j+1=b_{d_2-j+1}$, for all $1\leq j\leq d_2$. Since $\dim(\sigma_{c^\gamma})=\dim(\sigma_{b^\beta})$, if $c=b$ then $\gamma=\beta$. If $c>b$, then $\pi_2|_{\Sigma_{c^\gamma}}$ will have lower fiber dimension than $\pi_2|_{\Sigma_{b^\beta}}$. For every $\Lambda_2\in \pi_2(X)\cap \pi_2(\Sigma_{c^\gamma})$, by Proposition \ref{pushforward of class} and Corollary \ref{class of fiber}, the class of $(\pi_1\circ\pi_2|_{\Sigma_{c^\gamma}})(\Lambda_2)$ has lower dimension than the Poincar\'e dual of the class of $(\pi_1\circ\pi_2|_{\Sigma_{a^\alpha}})(\Lambda_2)$. Therefore $(\pi_1\circ\pi_2|_{\Sigma_{c^\gamma}})(\Lambda_2)\cap (\pi_1\circ\pi_2|_{\Sigma_{c^\gamma}})(\Lambda_2)=\emptyset$ for any $\Lambda_2\in \pi_2(X)\cap \pi_2(\Sigma_{c^\gamma})$, which implies $X\cap \Sigma_{c^\gamma}=\emptyset$. We conclude that $X$ has class $\sigma_{a^\alpha}$.

Case II: $\alpha_i=1$ and $a_{i+1}=a_i+1$.
\begin{Ex}
Consider the Schubert class $\sigma_{2^1,3^2,4^1}$ in $F(2,3;4)$. Under the projection $\pi_1:F(2,3;4)\rightarrow G(2,4)$, it is taken to the class $\sigma_{2,4}$ where $2$ is not rigid. The class $\sigma_{2,4}$ in $G(2,4)$ can be represented by a hyperplane section $H$ of $G(2,4)$ in the Pl\"ucker embedding. Define
$$X:=\{(\Lambda_1,\Lambda_2)\in F(2,3;4)|\Lambda_1\in H,\Lambda_1\subset \Lambda_2\subset F_4\}.$$
$X$ is the preimage of $H$ under $\pi_1$. It is easy to see that $X$ is a hyperplane section of $F(2,3;4)$, whose class is given by $\sigma_{2^1,3^2,4^1}$.
\end{Ex}
Construction of Case II: By assumption, we must have $\alpha_{i+1}=2$, $a_{i+1}=a_i+2$ and $\alpha_{i+2}=1$. Let $a'$ be the subsequence obtained from $a$ by removing all the sub-indices with upper index $2$. Since $\alpha_i=1$, there is a sub-index $i'$ such that $a_i=a'_{i'}$. A same construction as in Remark \ref{ex of non rigid} gives a non-Schubert subvariety $Y$ of $G(d_1,n)$ with class $[Y]=\sigma_{a'}\in A(G(d_1,n))$ such that $$\dim(\Lambda\cap F_{a'_j})\geq j,\ \text{ for all }\Lambda\in Y\text{ and }j\neq i',$$
for some $(d_1-1)$-step partial flag 
$$F_{a'_1}\subset...\subset F_{a'_{i'-1}}\subset F_{a'_{i'+1}}\subset ...\subset F_{a'_{d_1}}.$$
Complete the above partial flag to a partial flag $$F_{a_1}\subset...\subset F_{a_{i}-2}\subset F_{a_{i+1}}\subset ...\subset F_{a_{d_2}}.$$

Define
$$X:=\{(\Lambda_1,\Lambda_2)|\Lambda_1\in Y,\Lambda_1\subset \Lambda_2,\dim(F_{a_j}\cap \Lambda_2)\geq j,\ \text{ for }a_j\neq a_i-1,a_i,a_i+1\}.$$
The fiber of $\pi_1|_X$ at a point $\Lambda_1\in \pi_1(X)$ is the Schubert varieties defined by the partial flag that consists of $\Lambda_1$ and the span of $\Lambda_1$ with $F_{a_j}$, $\alpha_j=2$. By the fiber dimension theorem, $X$ is irreducible. By intersecting the $X$ with Schubert varieties of complementary dimension, we conclude that $[X]=\sigma_{a^\alpha}$.

More generally, given a Schubert class $\sigma_{a^\alpha}$ in $F=F(d_1,...,d_k;n)$ and assume the sub-index $a_i$ is essential but not rigid. If $a_{i+1}=a_i+1$, then set $I=\alpha_{i+1}$. Otherwise, set $I=k$. The construction in Remark \ref{ex of non rigid} gives a non-Schubert subvariety $Y$ in $G(d_{I},n)$ such that $$\dim(F_{a_j}\cap \Lambda_I)\geq \mu_{j,I},\text{ for all }j\neq i\text{ such that } \alpha_j\leq I, \forall\Lambda_I\in Y$$
for some partial flag
$$F_{a_1}\subset...\subset F_{a_i-2}\subset F_{a_i+2}\subset...\subset F_{a_{d_k}}.$$
Define
$$X:=\{(\Lambda_1,...,\Lambda_k)\in F|\Lambda_I\in Y,\dim(F_{a_j}\cap \Lambda_s)\geq \mu_{j,s}\text{ for }s\neq I,\text{ and }a_j\neq a_i-1,a_i,a_i+1\}.$$
Then a similar argument as above shows that $X$ has class $\sigma_{a^\alpha}$ but is not a Schubert variety.

Combining Theorem \ref{rigid index in g} and Theorem \ref{rigid index in f}, we obtain the following classification of rigid sub-indices in partial flag varieties:
\begin{Cor}
Given a Schubert index $a^\alpha=(a_1^{\alpha_1},...,a_{d_k}^{\alpha_{d_k}})$ in $F(d_1,...,d_k;n)$. An essential sub-index $a_i$ is rigid if and only if one of the following holds:
\begin{enumerate}
\item $a_{i+1}-a_i\geq 3$;
\item $a_{i+1}-a_i=2$ and 
\begin{enumerate}
\item either $a_i-a_{i-1}=1$; or
\item $\alpha_i<\alpha_{i+1}$
\end{enumerate}
\item $a_{i+1}-a_i=1$, and 
\begin{enumerate}
\item either $a_{i+2}-a_i\geq 3$; or
\item $\alpha_{i+2}>\alpha_i$; or
\item $a_{i-1}=a_i-1$ and $\alpha_{i-1}<\alpha_{i+1}$.
\end{enumerate}
\end{enumerate}
\end{Cor}
\begin{proof}
If (1), (2), (3-a) or (3-b) holds, then by Theorem \ref{rigid index in g} and Proposition \ref{pushforward of class}, $a_i$ is a rigid sub-index with respect to the Schubert class $(\pi_{\alpha_i})_*(\sigma_{a^\alpha})$ in $G(d_{\alpha_i},n)$.

If (3-c) holds, let $b=\max\{\alpha_{i-1},\alpha_i\}$. Then $a_i$ is rigid with respect to the Schubert class $(\pi_{{b}})_*(\sigma_{a^\alpha})$ in $G(d_{{b}},n)$.

Conversely, if $a_{i+1}-a_i=2$ and both (2-a) and (2-b) fail, then $a_i$ is not rigid with respect to $(\pi_j)_*(\sigma_{a^\alpha})$ for all $j\geq \alpha_i$.

If $a_{i+1}-a_i=1$, and (2-a), (2-b) and (2-c) fail, then for each $j\geq \alpha_i$, either $a_i$ is not essential or $a_i$ is not rigid with respect to $(\pi_j)_*(\sigma_{a^\alpha})$.

\end{proof}
As an application of the previous result, we classify the rigid Schubert classes in partial flag varieties. Given a Schubert class, if all of the essential sub-indices are rigid, then by Theorem \ref{rigid index in f}, for every representative $X$, there exists a unique set of subspaces $E=\{F_{a_i}\}_{a_i\text{essential}}$ that imposes the corank conditions on $X$. If furthermore the subspaces in $E$ form a partial flag, then $X$ is a Schubert variety.
\begin{Rem}
It is not always true that the subspaces in $E$ form a partial flag. For example, consider the partial flag variety $F(1,3;4)$. Let $\pi_1:F(1,3;4)\rightarrow G(1,4)$ and $\pi_2:F(1,3;4)\rightarrow G(3,4)$ be the canonical projections. Let $F_1$ be a one-dimensional linear subspace, and $F_3$ be a three-dimensional linear subspace that does not contain $F_1$. Let $Y$ be the Schubert variety in $G(3,4)$ defined by
$$Y:=\{\Lambda\in G(3,4)|F_1\subset \Lambda\}.$$
Define
$$X:=\{(\Lambda_1,\Lambda_2)\in F(1,3;4)|\Lambda_2\in Y,\Lambda_1\subset(\Lambda_2\cap F_3)\}.$$
By assumption, $F_3\notin Y$. For every $\Lambda\in Y$, $\dim(\Lambda\cap F_3)=2$. Therefore every fiber of $\pi_2|_X$ is isomorphic to $\mathbb{P}^1$. In particular, all the fibers are irreducible and have the same dimension. By the theorem of fiber dimension, $X$ is irreducible. By intersecting $X$ with Schubert varieties of complementary dimension, the class of $X$ is given by $\sigma_{1^2,3^1,4^2}$. Now the sub-indices $1$ and $3$ are rigid, which correspond to $F_1$ and $F_3$ respectively, but $F_1$ is not contained in $F_3$ by assumption.
\end{Rem}
The following lemma gives a direction when we should expect the subspaces in $E$ form a partial flag:
\begin{Lem}\label{coincide}
Let $\sigma_{a}$ be a Schubert class in $G(k,n)$. Let $X$ be a representative of $\sigma_a$. Assume there are two linear subspaces $F_{a_i}$ and $F_{a_j}$ of dimension $a_i$ and $a_j$ respectively, such that $\dim(\Lambda\cap F_{a_i})\geq i$, $\dim(\Lambda\cap F_{a_j})\geq j$ for all $\Lambda\in X$. If $i<j$ and $a_j$ is essential, then $F_{a_i}\subset F_{a_j}$.
\end{Lem}
\begin{proof}
We use induction on $a_i-i$. If $a_i=i$, suppose for a contradiction that $F_{a_i}\not\subset F_{a_j}$. Take a one-dimensional linear subspace $G_1$ in $F_{a_i}$, which is not contained in $F_{a_j}$. Let $W$ be the span of $G_1$ and $F_{a_j}$. Then $\dim(W)=a_j+1$ and for every $\Lambda\in X$, $\dim(\Lambda\cap W)\geq j+1$. This contradicts the condition that $a_j$ is essential.

Now assume $a_i-i>0$ and the statement is true whenever $a_{i'}-i'<a_i-i$. Consider the incidence correspondence
$$I:=\{(\Lambda,H)|\Lambda\in X, H\text{ is a hyperplane containing }\Lambda\}\subset F(k,n-1;n).$$
Let $\pi_2:I\rightarrow G(n-1,n)\cong (\mathbb{P}^{n-1})^*$ be the second canonical projection. Let $Y=\pi_2(I)$ be the image of $I$ under $\pi_2$. For every $H\in Y$, define $X_H:=\{\Lambda\in X|\Lambda\subset H\}$. By assumption, $X_H$ is non-empty. For a general $H\in Y$, by Corollary \ref{class of fiber} and Proposition \ref{pushforward of class}, the class of $X_H$ in $G(k,n)$ is given by $\sigma_{a'}$, where $a'_t=t$ if $a_t=t$, $a'_t=a_t-1$ if $a_t>t$. Let $F_{a_i}^H:=F_{a_i}\cap H$ and $F_{a_j}^H:=F_{a_j}\cap H$. By induction, $F_{a_i}^H\subset F_{a_j}^H$. By assumption, $\dim(X)\geq 1$ and therefore $\dim(Y)\geq 1$. Thus the hyperplanes parametrized by $Y$ cover the entire space, in particular, they cover $F_{a_j}$ and $F_{a_i}$. As varying $H\in Y$, we conclude that $F_{a_i}\subset F_{a_j}$. 
\end{proof}

This leads to the following definition.
\begin{Def}\label{link}
Let $\sigma_{a^\alpha}$ be a Schubert class in $F(d_1,...,d_k;n)$. We define a relation `$\rightarrow$' between two sub-indices: $a_i\rightarrow a_j$ if $i<j$ and $a_j$ is essential in $(\pi_t)_*(\sigma_{a^\alpha})$ for some $t\geq \min(\alpha_i,\alpha_j)$. This relation extends to a strict partial order (which we also denote by `$\rightarrow$') on the set $A=\{a_i\}_{i=1}^{d_k}$ by transitivity. If $a_i\rightarrow a_j$, then we say $a_i$ is linked to $a_j$.
\end{Def}
We then obtain the following classification of rigid Schubert classes in $F(d_1,...,d_k;n)$:

\begin{Thm}\label{rigid in f}
A Schubert class $\sigma_{a^\alpha}\in A(F(d_1,...,d_k;n))$ is rigid if and only if all essential sub-indices are rigid and the set of all essential sub-indices is strict totally ordered under the relation `$\rightarrow$' defined in Definition \ref{link}.
\end{Thm}
\begin{proof}
If all essential sub-indices are rigid and the set of all essential sub-indices is totally ordered by `$\rightarrow$', then by Theorem \ref{rigid index in f} and Lemma \ref{coincide}, every representative of $\sigma_{a^\alpha}$ is a Schubert variety defined by the partial flag obtained from Theorem \ref{rigid index in f}.

Conversely, if one of the essential sub-indices is not rigid, then we have constructed a non-Schubert subvariety that represents the class $\sigma_{a^\alpha}$. If there are two essential sub-indices $a_i$ and $a_j$, $i<j$ such that $a_i$ is not linked to $a_j$, assume there is no other essential $a_{r}$ between $i$ and $j$. Take a partial flag
$$F_{a_1}^{\alpha_1}\subset...\subset F_{a_{d_k}}^{\alpha_{d_k}}.$$
Replace $F_{a_j}$ with another linear subspace $G_{a_j}$ of the same dimension such that $$F_{a_{i-1}}\subset G_{a_j}\subset F_{a_{j+1}},$$
and $\dim(F_{a_{i}}\cap G_{a_j})=a_{i}-1$.

Define 
$$X:=\{\Lambda\in F(d_1,...,d_k)|\dim(G_{a_j}\cap \Lambda_t)\geq \mu_{j,t},\dim(F_{a_s}\cap \Lambda_t)\geq \mu_{s,t}, \text{ for }s\notin (i,j]\}.$$
Then $X$ represents the Schubert class $\sigma_{a^\alpha}$ but is not a Schubert variety since it is not defined by a partial flag.
\end{proof}

\section{The rigidity problem in orthogonal partial flag varieties}\label{type BD}
In this section, we study the rigidity problem in orthogonal partial flag varieties (type B or type D), which generalizes orthogonal Grassmannians. We first review the rigidity results in orthogonal Grassmannians.
\subsection{Orthogonal Grassmannians}We start with the definition of essential sub-indices:
 \begin{Def}\label{def1}
Let $\sigma_{a;b}$ be a Schubert class in $OG(k,n)$ defined in Section \ref{og}. A sub-index $a_i$ is called {\em essential} if one of the following holds:
\begin{itemize}
\item $i<s$ and $a_{i}<a_{i+1}-1$
\item $n$ is odd and $i=s$;
\item $n$ is even, $i=s$ and $a_s+b_{k-s}\neq n-2$.
\end{itemize}

A sub-index $b_j$ is called {\em essential} if either $j=1$ or $b_{j}\neq b_{j-1}+1$. 
\end{Def}
\begin{Prop}
Let $X$ be a representative of $\sigma_{a;b}$ in $OG(k,n)$. If there exists a flag $F_\bullet$ of isotropic subspaces such that for all $\Lambda\in X$, $\dim(F_{a_i}\cap \Lambda)\geq i$ for all essential $a_i$, $\dim(F_{b_j}^\perp\cap \Lambda)\geq k-j+1$ for all essential $b_j$, then $X$ has to be the Schubert variety defined by $F_\bullet$.
\end{Prop}
\begin{proof}
First assume $n$ is odd. The statement then follows from the definition and the observation that if $a_i=a_{i+1}-1$, then $\dim(\Lambda\cap F_{a_{i+1}})\geq i+1$ implies $\dim(\Lambda\cap F_{a_{i}})\geq i$, and if $b_j=b_{j-1}+1$, then $F^\perp_{b_j}$ is of codimension 1 in $F^\perp_{b_{j-1}}$ and therefore $\dim(\Lambda\cap F_{b_{j-1}}^\perp)\geq k-j+2$ implies $\dim(\Lambda\cap F_{b_{j}}^\perp)\geq k-j+1$.

The case when $n$ is even is slightly different when $a_s=\frac{n}{2}$ or $b_{k-s}=\frac{n}{2}-1$, due to the convention that we use $F_{n/2-1}^\perp$ to denote, instead of the orthogonal component of $F_{n/2-1}$, the maximal isotropic subspace containing $F_{n/2-1}$ other than $F_{n/2}$. 

If $a_s=b_{k-s}=\frac{n}{2}-1$, then since $F_{n/2-1}$ is of codimension 1 in $F_{n/2-1}^\perp$, $\dim(\Lambda\cap F_{n/2-1}^\perp)\geq s+1$ implies $\dim(\Lambda\cap F_{n/2-1})\geq s$. 

If $a_s=\frac{n}{2}$ and $b_{k-s}=\frac{n}{2}-2$, then the condition $\dim(\Lambda\cap F_{n/2-2}^\perp)\geq s+1$ implies either $\dim(\Lambda\cap F_{n/2})\geq s$ or $\dim(\Lambda\cap F_{n/2-1}^\perp)\geq s$. Each possibility gives one Schubert variety. Since $X$ is irreducible, $X$ must be one of the Schubert varieties.

If $b_{k-s}=\frac{n}{2}-1$ and $b_{k-s-1}=\frac{n}{2}-2$, then the condition $\dim(\Lambda\cap F_{n/2-2}^\perp)\geq s+1$ implies either $\dim(\Lambda\cap F_{n/2})\geq s$ or $\dim(\Lambda\cap F_{n/2-1}^\perp)\geq s$. Each possibility gives one Schubert variety. Since $X$ is irreducible, $X$ must be one of the Schubert varieties.
\end{proof}

\begin{Def}
An essential sub-index $a_i$ is called {\em rigid} if for every subvariety $X$ of $OG(k,n)$ representing $\sigma_{a;b}$, there exists an isotropic subspace $F_{a_i}$ of dimension $a_i$ such that $$\dim(\Lambda\cap F_{a_i})\geq i,\ \ \ \forall\Lambda\in X.$$

An essential sub-index $b_j$ is called {\em rigid} if for every subvariety $X$ representing $\sigma_{a;b}$, there exists an isotropic subspace $F_{b_j}$ of dimension $b_j$ such that $$\dim(\Lambda\cap F_{b_j}^\perp)\geq k-j+1,\ \ \ \forall\Lambda\in X.$$
\end{Def}

In \cite[Theorem 1.6]{YL}, we classified the rigid sub-indices in orthogonal Grassmannians:
\begin{Thm}\cite[Theorem 1.6]{YL}\label{rigid index og}
Let $\sigma_a^b$ be a Schubert class in $OG(k,n)$. An essential sub-index $a_i$ is not rigid if and only if one of the following holds:
\begin{enumerate}
\item $a_i\neq b_j$ for all $1\leq j\leq k-s$, $a_i-a_{i-1}\geq 2$ and $$a_{i+1}-a_i=2+\#\{j|a_i< b_j<a_{i+1}\};$$
\item $a_i=b_j$ for some $j$ and $$\#\{\mu|a_\mu\leq b_j\}=k-j+b_{j}-\frac{n-3}{2}.$$
\end{enumerate}

An essential sub-index $b_j$ is rigid if and only if either $b_j=0$ or there exist $1\leq i\leq s$ and $j\leq j'\leq k-s$ such that $a_i=b_{j'}$ and $$\#\{\mu|a_{\mu}\leq b_{j'}\}>k-j'+b_{j'}-\frac{n-3}{2}.$$
\end{Thm}

In particular, we obtained a classification of rigid Schubert classes in orthogonal Grassmannians:
\begin{Thm}\cite[Theorem 1.7]{YL}\label{rigid in og}
Let $\sigma_a^b$ be a Schubert class in $OG(k,n)$. Let $b_\gamma$ be the largest essential sub-index in $b=(b_1,...,b_{k-s})$. Then $\sigma_a^b$ is rigid if and only if all of the following conditions hold:
\begin{enumerate}
\item $b_\gamma=a_i$ for some $1\leq i\leq s$ and $$\#\{\mu|a_{\mu}\leq b_{\gamma}\}>k-\gamma+b_{\gamma}-\frac{n-3}{2};$$
\item there is no $1\leq i\leq s$ such that $a_i\neq b_j$ for all $1\leq j\leq s$, $a_i-a_{i-1}\geq 2$ and $$a_{i+1}-a_i=2+\#\{j|a_i< b_j<a_{i+1}\}.$$
\end{enumerate}
\end{Thm}

\begin{Rem}\label{exinof}
In \cite{YL}, we constructed the counter-examples when $a_i$ or $b_j$ is not rigid. For reader's convenience, we give a brief description of those constructions.

First assume $b_j$ is not rigid. There are two cases: either $a_i\neq b_j'$ for all $1\leq i\leq s$ and $j\leq j'\leq k-s$, or  $a_i=b_{j'}$ for some $1\leq i\leq s$ and $j\leq j'\leq k-s$ and $\#\{\mu|a_\mu\leq b_{j'}\}=k-j'+b_{j'}-\frac{n-3}{2}.$ In the first case, consider the restriction variety $X$ defined by the following sequence
$$F_{a_1}\subset...\subset F_{a_s}\subset Q_{n-b_{k-s}}^{b_{k-s}-1}\subset...\subset Q_{n-b_j}^{b_{j}-1}\subset Q_{n-b_{j-1}}^{{b_{j-1}}}\subset...\subset Q_{n-b_1}^{b_1}.$$
The $X$ is a representative of $\sigma_{a;b}$ but is not a Schubert variety.

If $a_i=b_{j'}$ for some $1\leq i\leq s$ and $j\leq j'\leq k-s$ and $\#\{\mu|a_\mu\leq b_{j'}\}=k-j'+b_{j'}-\frac{n-3}{2},$ let $j_0$ be the minimal one among such $j'$. Consider the restriction variety $X$ defined by the sequence
$$F_{a'_1}\subset...\subset F_{a'_s}\subset Q_{n-b_{k-s}}^{b_{k-s}-1}\subset...\subset Q_{n-b_j}^{b_{n_j}-1}\subset Q_{n-b_{j-1}}^{{b_{j-1}}}\subset...\subset Q_{n-b_1}^{b_1},$$
where $a'$ is obtained from $a$ by changing all the $a_i$ such that $a_i\geq b_{j_0}$ to the maximal admissible sequence with respect to $b_{k-s}-1,....,b_{j_0}-1$ (i.e. the sequence consists of all the numbers from $b_{j_0}+1$ to $\left[\frac{n}{2}\right]$ excluding the numbers equal to $b_{j'}$). The $X$ is a non-Schubert representative of $\sigma_{a;b}$.

Now assume $a_i$ is not rigid. There are also two cases: either $a_i=b_j$ for some $j$ and $\#\{\mu|a_\mu\leq b_j\}=k-j+b_{j}-\frac{n-3}{2},$ or $a_i\neq b_j$ for all $1\leq j\leq k-s$, $a_i-a_{i-1}\geq 2$ and $a_{i+1}-a_i=2+\#\{j|a_i< b_j<a_{i+1}\}.$ In the first case, $b_j$ is not rigid, and the previous construction also works for this case.

If $a_i\neq b_j$ for all $1\leq j\leq k-s$, $a_i-a_{i-1}\geq 2$ and $a_{i+1}-a_i=2+\#\{j|a_i< b_j<a_{i+1}\}$, then the sequence $b$ must contain $a_i+1,...,a_{i+1}-2$. Say $b_\gamma=a_i+1$. First assume that $\gamma=1$. Choose a complete isotropic flag $F_\bullet$. In $OG(k-z_i,n)$, consider the Schubert index $(a';b')$ defined by $a'_\mu=a_\mu+z_i$ for $1\leq \mu\leq i$, $a'_{\mu}=a_{\mu}$ for $i+1\leq \mu\leq s$, $b'_j=b_{j+z_i}$ for $1\leq j\leq k-s-z_i$. Then $a'_i=a_i+z_i$ is not rigid. Let $Y$ be a subvariety of $OG(k-z_i,n)$, but not a Schubert variety, such that $[Y]=\sigma_{a'}^{b'}$, $\dim(\Lambda\cap F_{a'_\mu})\geq \mu$ for $\mu\neq i$, $\dim(\Lambda\cap F_{b'_j}^\perp)\geq k-j+1$ for $1\leq j\leq k-s-z_i$ and for all $\Lambda\in Y$. Let $X$ be the Zariski closure in $OG(k,n)$ of the following locus of $k$-planes:
$$\{\text{span}\{G_{z_i},\Lambda\}|G_{z_i}\text{ is a linear subspace contained in } F_{a_i+1}^\perp\backslash F_{a_{i+1}}^\perp, \Lambda\in Y,\Lambda\subset G_{z_i}^\perp\}. $$
Then $X$ is a non-Schubert variety representing the class $\sigma_{a;b}$.

Now assume $\gamma\geq 2$. Let $Z$ be a non-Schubert subvariety of $OG(k-\gamma+1,n)$ representing $\sigma_{a'}^{b'}$, where $a'=a$, $b'_j=b_{j+\gamma-1}$ for $1\leq j\leq k-s-\gamma+1$, such that there does not exist a linear space of dimension $a_i$ that meet all $k-\gamma+1$-planes parametrized by $Z$ in dimension at least $i$. Let $X$ be the Zariski closure of 
$$\{\Lambda\in OG(k,n)|\dim(\Lambda\cap F_{b_j}^\perp)\geq k-j+1, 1\leq j<\gamma,\Lambda\cap F_{b_\gamma}\in Y\}.$$
Then $X$ is a non-Schubert variety representing the class $\sigma_{a;b}$.
\end{Rem}

\subsection{Classes that arise from the canonical projections}
Consider the canonical projections $\pi_i:OF(d_1,...,d_k;n)\rightarrow OG(d_i,n)$. The images of Schubert varieties are again Schubert varieties. In this section, we calculate the class of the push-forward cycle, as well as the class of a general fiber.
\begin{Prop}\label{class of pushforward in of}
Let $X$ be a subvariety in $OF(d_1,...,d_k;n)$ representing the Schubert class $\sigma_{a^\alpha;b^\beta}$. Then 
$$[\pi_i(X)]=\sigma_{a'^{\alpha'};b'^{\beta'}},$$
where $(a'^{\alpha'},b'^{\beta'})$ is obtained from $(a^\alpha;b^\beta)$ by removing the elements whose upper index is greater than $i$.
\end{Prop}
\begin{proof}
The proof is identical to the proof of Proposition \ref{pushforward of class}.
\end{proof}

\begin{Prop}
Let $X$ be a representative of $\sigma_{a^\alpha;b^\beta}$. A general fiber of $\pi_k|_X$ has class $\sigma_{c^\mu}$, where $c^\mu=(1^{\alpha_1},...,s^{\alpha_s},(s+1)^{\beta_{d_k-s}},...,d_k^{\beta_1})$.
\end{Prop}
\begin{proof}
The proof is identical to the proof of Proposition \ref{fiberk}.
\end{proof}

More generally, we have

\begin{Cor}\label{class of pullback in of}
Let $X$ be a representative of $\sigma_{a^\alpha;b^\beta}$. Then a general fiber of $\pi_m|_X$, $1\leq m\leq k$, has class $\sigma_{a'^{\alpha'},b'^{\beta'}}$, where $(a'^{\alpha'},b'^{\beta'})$ is obtained as follows. Set
$$y_{j,m}:=\#\{p|a_p>b_j,\alpha_p\leq m\},$$ 
$$z_{j,m}:=\#\{q|b_q\geq b_j,\beta_q\leq m\}$$
$$h_{i,m}:=\#\{r|b_r\geq a_i,\beta_r\leq m\}$$
 for $1\leq i\leq s$ and $1\leq j\leq k-s$. 
\begin{itemize}
\item If $\alpha_i\leq m$, replace $a_i$ with $\mu_{i,m}$ and put it in a set $A$; 
\item If $\alpha_i> m$, replace $a_i$ with $a_i+\mu_{s,m}-\mu_{i,m}+h_{i,m}$ and put it in the set $A$; 
\item If $\beta_j\leq m$, replace $b_j$ with $\nu_{j,m}$, and put it in the set $A$; 
\item If $\beta_j>m$, replace $b_j$ with $b_j+y_{j,m}+z_{j,m}$, and put it in a set $B$. 
\end{itemize}
The sequences $a'^{\alpha'}$ and $b'^{\beta'}$ are then obtained by arranging the elements in $A$ and $B$ resp. in an increasing order.

\end{Cor}

\begin{proof}
The class of a general fiber can be realized as an intersection product in $OF(d_1,...,d_k;n)$, and therefore there is no harm to assume $X$ to be a Schubert variety.

Assume $X$ is a Schubert variety defined by the following partial flag
$$F_{a_1}^{\alpha_1}\subset...\subset F_{a_s}^{\alpha_s}\subset (F_{b_{k-s}}^\perp)^{\beta_{k-s}}\subset...\subset (F_{b_1}^\perp)^{\beta_1}.$$
Let $\Lambda_m$ be a general point in $X_m:=\pi_m(X)$. 
Let $G$ be the sequence obtained from $F$ by the following procedures:
\begin{enumerate}
\item For every $1\leq i\leq s$ and $1\leq j\leq d_k$,
\begin{itemize}
\item If $\alpha_i\leq m$, replace $F_{a_i}$ by $\Lambda_m\cap F_{a_i}$;
\item If $\alpha_i>m$, replace $F_{a_i}$ by span$(\Lambda_m,\Lambda_m^\perp\cap F_{a_i})$;
\item If $\beta_j\leq m$, replace $F_{b_j}^\perp$ by $\Lambda_m\cap F_{b_j}^\perp$;
\item If $\beta_j>m$, replace $F_{b_j}^\perp$ by $(\text{span}(\Lambda_m,\Lambda_m^\perp\cap F_{b_j}))^\perp$.
\end{itemize}
\item Re-arrange the elements so that the dimension is in an increasing order.
\end{enumerate}

Then the fiber of $\pi_m|_X$ is the Schubert variety in $OF(d_1,...,d_k;n)$ defined by the flag $G$.

\end{proof}

\subsection{Orthogonal partial flag varieties}
In this section, we study the rigidity problem in orthogonal flag varieties which are described in Section \ref{of}.

\begin{Ex}
Consider the orthogonal flag variety $OF(1,2;5)$. There are two canonical projections $$\pi_1:OF(1,2;5)\rightarrow OG(1,5);$$ $$\pi_2:OF(1,2;5)\rightarrow OG(2,5).$$ 
Consider the Schubert class $\sigma_{1^1;1^2}$. The class $(\pi_2)_*(\sigma_{1^1;1^2})=\sigma_{1;1}\in A(OG(2,5))$ is not rigid by Theorem \ref{rigid in og}, nevertheless we claim that the Schubert class $\sigma_{1^1;1^2}$ is rigid in $OF(1,2;5)$.

Let $X$ be a representative of $\sigma_{1^1;1^2}$ in $OF(1,2;5)$. Notice that $(\pi_1)_*(\sigma_{1^1;1^2})=\tilde{\sigma}_1\in A(OG(1,5))$ is the class of a point, and is rigid by Theorem \ref{rigid in og}. Let $\pi_1(X)=\{F_1\}$. Then every isotropic 2-plane that contains $F_1$ must be contained in $F_1^\perp$. Therefore $X$ is contained in the Schubert variety 
$$\Sigma=\{(\Lambda_1,\Lambda_2)\in OF(1,2;5)|F_1\subset \Lambda_1,\Lambda_2\subset F_1^\perp\},$$
which has the same class of $X$, and in particular $\dim(X)=\dim(\Sigma)$. Since the Schubert varieties are irreducible, $X=\Sigma$ is a Schubert variety.

\end{Ex}

\begin{Def}\label{defofrigid}
Let $(a^\alpha,b^\beta)$ be a Schubert index for $OF(d_1,...,d_k;n)$. A sub-index $a_i$ or $b_j$ is called {\em essential} if it is essential with respect to the class $(\pi_t)_*(\sigma_{a^\alpha;b^\beta})$ in $OG(d_t,n)$ for some $1\leq t\leq k$.
\end{Def}

\begin{Prop}
Let $X$ be a representative of $\sigma_{a^\alpha;b^\beta}$ in $OF(d_1,...,d_k;n)$. If there exists a flag $F_\bullet$ of isotropic subspaces such that for all $(\Lambda_1,...,\Lambda_k)\in X$, $\dim(F_{a_i}\cap \Lambda_t)\geq \mu_{i,t}$ for all essential $a_i$, $\dim(F_{b_j}^\perp\cap \Lambda_t)\geq \nu_{j,t}$ for all essential $b_j$, $1\leq t\leq k$, then $X$ has to be a Schubert variety.
\end{Prop}
\begin{proof}
The proof similar to Proposition \ref{essential flag}.
\end{proof}

\begin{Def}
Let $(a^\alpha,b^\beta)$ be a Schubert index for $OF(d_1,...,d_k;n)$. An essential $a_i$ is called {\em rigid} if for every representative $X$ of $\sigma_{a^\alpha;b^\beta}$, there exists an isotropic subspace $F_{a_i}$ of dimension $a_i$ such that 
$$\dim(F_{a_i}\cap\Lambda_t)\geq \mu_{i,t}, \forall (\Lambda_1,...,\Lambda_k)\in X,\alpha_i\leq t\leq k.$$
An essential $b_j$ is called {\em rigid} if for every representative $X$, there exists an isotropic subspace $F_{b_j}$ of dimension $b_j$ such that
$$\dim(F_{b_j}^\perp\cap\Lambda_t)\geq \nu_{j,t}, \forall (\Lambda_1,...,\Lambda_k)\in X,\beta_j\leq t\leq k.$$
\end{Def}

Similar as in \S\ref{type A}, if the rank condition in Definition \ref{defofrigid} holds for one component, then it should hold for all the other components. To be more precise, 

\begin{Thm}\label{rigidindexinoff}
Let $\sigma_{a^\alpha;b^\beta}$ be a Schubert class in $OF(d_1,...,d_k;n)$. Let $X$ be a representative of $\sigma_{a^\alpha;b^\beta}$. Let $F_{a_i}$ and $F_{b_j}$ be isotropic subspaces of dimension $a_i$ and $b_j$ respectively. If for some $\alpha_i\leq t\leq k$, $$\dim(F_{a_i}\cap\Lambda_t)\geq \mu_{i,t}\text{ for all }(\Lambda_1,...,\Lambda_k)\in X,$$ then the same inequailty holds for all $\alpha_i\leq t'\leq k$.

Similarly, if for some $\beta_j\leq t\leq k$, $$\dim(F_{b_j}^\perp\cap\Lambda_t)\geq \nu_{j,t}\text{ for all }(\Lambda_1,...,\Lambda_k)\in X,$$ then the same inequality holds for all $\beta_j\leq t'\leq k$.
\end{Thm}

\begin{proof}
The proof of the case of $a_i$ is identical to the proof of Theorem \ref{rigid index in f}.

The case of $b_j$ can also be done similarly. If $n$ is even and $b_j=\frac{n}{2}-1$, then $F_{n/2-1}^\perp$ is the maximal isotropic subspace that contains $F_{n/2-1}$ and belongs to the different component than $F_{n/2}$. Since the two components of the space of maximal isotropic subspaces can be interchanged under an involution of quadrics, the case of $b_j=\frac{n}{2}-1$ can be treated in the same way as $a_i=\frac{n}{2}$.

Now assume $n$ is odd or $b_j<\frac{n}{2}-1$. Let $\Lambda_r$ be a general point in $\pi_r(X)$. Suppose, for a contradiction, that $$\dim(F_{b_j}^\perp\cap\Lambda_r)<\nu_{j,r}.$$ Let $W:= F_{b_j}^\perp\cap \Lambda_r$ and $w=\dim(W)$. Let $\sigma_{a';b'}$ be the class of $\pi_t(\pi_r^{-1}(\Lambda_r))$ in $OG(k_t,n)$, which can be obtained from Proposition \ref{class of pushforward in of} and Corollary \ref{class of pullback in of}.

If $r>t$, then $a'_{\mu_{j,t}}=\mu_{j,r}$. However, $\dim(W)<\mu_{j,r}$ and for all $\Lambda_t\in \pi_t(\pi_r^{-1}(\Lambda_r))$,
\begin{eqnarray}
\dim(W\cap \Lambda_t)&=&\dim(F_{b_j}^\perp\cap\Lambda_r)\cap\Lambda_t\nonumber\\
&=&\dim(F_{b_j}^\perp\cap\Lambda_t )\nonumber\\
&\geq&\mu_{j,t}\nonumber
\end{eqnarray}
We reach a contradiction.

Now we may assume $t=k$ and $r<t$. Set $$z_r:=\#\{e|e<j,\beta_e>r\}.$$
Observe that $d_k-d_r=\mu_{j,k}-\mu_{j,r}+z_r$. If $z_r=0$, then since $\Lambda_r$ is of codimension $d_k-d_r$ in $\Lambda_k$,
\begin{eqnarray}
\dim(F_{b_j}^\perp\cap\Lambda_k)\geq \mu_{j,k}\Rightarrow\dim(F_{b_j}^\perp\cap\Lambda_r)&\geq&\mu_{j,k}-(d_k-d_r)\nonumber\\
&=&\mu_{j,r}\nonumber
\end{eqnarray}
If $z_r\neq 0$, let $$l:=\max\{e|e<j,\beta_e>r\}.$$ Then by Corollary \ref{class of pullback in of}, $$b'_{z_r}=b_l+\mu_{l,r}-\#\{i|\alpha_i\leq r,a_i\leq b_l\}.$$
On the other hand, consider the following isotropic subspace $$U:=\text{span}(\Lambda_r,\Lambda_r^\perp\cap F_{b_j}).$$
It is easy to check that
\begin{eqnarray}
\dim(U)&=&b_j+\dim(\Lambda_r\cap F_{b_j}^\perp)-\dim(\Lambda_r\cap F_{b_j})\nonumber\\
&=&b_j+w-\dim(\Lambda_r\cap F_{b_j})\nonumber\\
&\geq&b_j+w-\#\{i|\alpha_i\leq r,a_i\leq b_j\}\nonumber
\end{eqnarray}
and
\begin{eqnarray}
\dim(\Lambda_k\cap U^\perp)&=&\dim(\Lambda_k\cap \Lambda_r^\perp\cap (\Lambda_r^\perp\cap F_{b_j})^\perp)\nonumber\\
&=&\dim(\Lambda_k\cap (\Lambda_r^\perp\cap F_{b_j})^\perp)\nonumber\\
&=&d_k-\dim(\Lambda_r^\perp\cap F_{b_j})+\dim(\Lambda_r^\perp\cap F_{b_j}\cap \Lambda_k^\perp)\nonumber\\
&=&d_k-(b_j-d_r+\dim(\Lambda_r\cap F_{b_j}^\perp))+(b_j-d_k+\dim(\Lambda_k\cap F_{b_j}^\perp))\nonumber\\
&=&d_r-w+\mu_{j,k}\nonumber\\
&\geq&d_r-\mu_{j,r}+\mu_{j,k}+1\nonumber\\
&=&d_k-z_r+1.\nonumber
\end{eqnarray}
Hence
\begin{eqnarray}
\dim(U)+\dim(\Lambda_k\cap U^\perp)&\geq&b_j+d_r+\mu_{j,k}-\#\{i|\alpha_i\leq r,a_i\leq b_j\}\nonumber\\
&=&b_j+d_k+\mu_{j,r}-z_r-\#\{i|\alpha_i\leq r,a_i\leq b_j\}\nonumber\\
&>&b_l+\mu_{l,r}-\#\{i|\alpha_i\leq r,a_i\leq b_l\}+d_k-z_r+1.\nonumber
\end{eqnarray}
We reach a contradiction. Therefore $\dim(F_{b_j}^\perp\cap\Lambda_r)\geq\mu_{j,r}$ for all $r\geq \beta_j$.
\end{proof}

If an essential sub-index $a_i$ or $b_j$ is rigid in the $t$-th component, then the conditions in Theorem \ref{rigidindexinoff} automatically hold. However, unlike the case of type A flag varieties in \S\ref{type A}, this condition is sufficient but not necessary. This can be seen from the following two observations in orthogonal Grassmannians.

First, if $a_i$ is rigid and $a_i=b_j$, then we should expect $b_j$ is also rigid, and vice versa.

\begin{Lem}\label{atob}
Let $\sigma_{a;b}$ be a Schubert class in $OG(k,n)$, and let $X$ be a representative of $\sigma_{a;b}$. Assume $a_i=b_j\neq \frac{n}{2}-1$ for some $1\leq i\leq s$ and $1\leq j\leq k-s$. Let $F_{a_i}$ be an isotropic subspace of dimension $a_i$. Then
$$\dim(F_{a_i}\cap \Lambda)\geq i,\forall \Lambda\in X$$
if and only if 
$$\dim(F_{a_i}^\perp\cap \Lambda)\geq k-j+1,\forall \Lambda\in X.$$
\end{Lem}
\begin{proof}
First assume $j=1$ and $a_i=b_1$. We use induction on $a_i-i$. If $a_i=i$, then $F_i\subset \Lambda$ implies $\Lambda\subset F_i^\perp$. Conversely, if $\Lambda\subset F_i^\perp$ for all $\Lambda\in X$, then for each $p\in \mathbb{P}(F_{i}^\perp)\cap Q$, define $X_p:=\{\Lambda\in X|p\in\mathbb{P}(\Lambda)\}$. By specializing $X$ to a Schubert variety, we obtain if $p\in F_i$, then $[X_p]=\sigma_{a;b}=[X]$, which implies $X_p=X$. Therefore $F_i\subset \Lambda$ for all $\Lambda\in X$. 
If $a_i-i>0$, consider the incidence correspondence
$$I:=\{(\Lambda,H)|\Lambda\in X, X\subset H,H\text{ is a hyperplane}\}.$$
Let $H$ be a general point in $\pi_2(I)$. Define $X_H:=\{\Lambda\in X|\Lambda\subset H\}$. By \cite[Lemma 6.2]{Coskun2014RigidityOS}, the projective linear spaces parametrized by $X$ sweep out a quadric $Q_X$ of dimension $n-a_i-2$. Set $r:=corank(Q_X)$ and $Q_H:=Q_X\cap \mathbb{P}(H)$. Then $Q_H$ is a quadric of dimension $n-a_i-3$ with corank $r-1$ if $r\geq 1$ and is smooth if $r=0$. Notice that $Q_H$ is contained in a smooth quadric of dimension $n-4$ and $\Lambda\subset Q_H$ for all $\Lambda \in X_H$. Therefore $X_H$ is a subvariety of $OG(k,n-2)$. The class of $X_H$ in $OG(k,n-2)$ is given by $\sigma_{a',b'}$, where $a'_i=a_i-1$ and $b'_1=b_1-1$. Set $F_H=F_{a_i}\cap H$. If $\dim(F_{a_i}\cap \Lambda)\geq i$ for all $\Lambda\in X$, then $\dim(F_{H}\cap \Lambda)\geq i$ for all $\Lambda\in X_H$. By induction on $a_i-i$, the projective linear spaces parametrizes by $X_H$ sweep out the quadric $Q_H$ and $F_H$ is the singular locus of $Q_H$. This implies $r=\dim(F_H)+1=a_i$ and $Q_X=\mathbb{P}(F_{a_i}^\perp)\cap Q$. Conversely, if $Q_X=\mathbb{P}(F_{a_i}^\perp)\cap Q$. Then $Q_H$ has corank $a_i-1$ and the singular locus $F_{a_i}\cap H=F_H$. By induction, $\dim(F_{H}\cap \Lambda)\geq i$ for all $\Lambda\in X_H$. As varying $H$, $X_H$ cover $X$ and therefore $\dim(F_{a_i}\cap \Lambda)\geq i$ for all $\Lambda\in X$.

Now assume $j>1$. We use induction on $j$. Let $p$ be a general point in $Q_X$. Then the class of $X_p$ is given by $\sigma_{a';b'}$ where $a'_i=a_i+1$ if $a_i\leq b_1$, $a'_i=a_i$ if $a_i>b_1$, and $b'_t=b_{t+1}$ for $1\leq t\leq k-s-1$. Since $a_i=b_j>b_1$, $a'_i=a_i=b'_{j-1}$. Set $F_{a_i}^p:=\text{span}(p,p^\perp\cap F_{a_i})$. $\dim(F_{a_i}^p)=a_i$. If $\dim(F_{a_i}\cap \Lambda)\geq i, \forall \Lambda\in X$, then for every $\Lambda\in X_p$, 
\begin{eqnarray}
\dim(\Lambda\cap F_{a_i}^p)&=&\dim(\Lambda\cap\text{span}(p,p^\perp\cap F_{a_i}))\nonumber\\
&=&1+\dim(\Lambda\cap F_{a_i})\nonumber\\
&\geq&i+1\nonumber
\end{eqnarray}
By induction, $\dim(\Lambda \cap (F_{a_i}^p)^\perp)\geq k-(j-1)+1=k-j+2$. Notice that
$$F_{a_i}^\perp\cap (F_{a_i}^p)^\perp=F_{a_i}^\perp\cap p^\perp\cap (p^\perp\cap F_{a_i})^\perp=F_{a_i}^\perp\cap p^\perp,$$
which is a codimension one subspace in $(F_{a_i}^p)^\perp$. Therefore 
\begin{eqnarray}
\dim(\Lambda\cap F_{a_i}^\perp)&\geq&\dim(\Lambda\cap (F_{a_i}^p)^\perp)-1\nonumber\\
 &\geq&n-j+1.\nonumber
\end{eqnarray}
Conversely, if $\dim(F_{a_i}^\perp\cap \Lambda)\geq k-j+1,\forall \Lambda\in X$, then 
\begin{eqnarray}
\dim(\Lambda\cap (F_{a_i}^p)^\perp)&=&\dim(\Lambda\cap p^\perp\cap(p^\perp\cap F_{a_i})^\perp)\nonumber\\
&=&\dim(\Lambda\cap (p^\perp\cap F_{a_i})^\perp)\nonumber\\
&\geq&\dim(\Lambda\cap p)+\dim(\Lambda\cap F_{a_i}^\perp)\nonumber\\
&\geq&k-j+2.\nonumber
\end{eqnarray}
By induction, $\dim(F_{a_i}^p\cap \Lambda)\geq i+1,\forall \Lambda\in X_p$. Since $$F_{a_i}\cap F_{a_i}^p=F_{a_i}\cap\text{span}(p,p^\perp\cap F_{a_i})=p^\perp\cap F_{a_i},$$
$\dim(F_{a_i}\cap \Lambda)\geq \dim(F_{a_i}^p\cap \Lambda)-1\geq i,\forall \Lambda\in X_p$. As varying $p$, we conclude that $\dim(F_{a_i}\cap \Lambda)\geq i,\forall \Lambda\in X$.

\end{proof}

Second, if $b_j$ is rigid, then all essential $b_{j'}$, $j'<j$, are also rigid.

\begin{Lem}\label{rankbound}
Let $X$ be a subvariety representing the Schubert class $\sigma_{a;b}$ in $OG(k,n)$. Assume there exists an isotropic subspace $F_{b_{j}}$, $b_j\neq\frac{n}{2}-1$, such that 
$$\dim(\Lambda\cap F_{b_j}^\perp)\geq k-j+1,\forall \Lambda\in X.$$
Then for every essential $b_{j'}<b_j$, there exists an isotropic subspace $F_{b_{j'}}\subset F_{b_j}$ such that
$$\dim(\Lambda\cap F_{b_{j'}}^\perp)\geq k-j'+1,\forall \Lambda\in X.$$
\end{Lem}
\begin{proof}
We use induction on $j'$. If $j'=1$, by \cite[Lemma 6.2]{Coskun2014RigidityOS}, the projective linear spaces $\mathbb{P}^{n-1}$ parametrized by $X$ sweeps out a quadric $Q_X$ of dimension $n-b_1-2$. On the other hand, the intersection of $\mathbb{P}(F_{b_j}^\perp)$ and $Q$ is a sub-quadric $Q_{n-b_j}^{b_j}$, whose dimension equals $n-b_j-2$ and corank equals $b_j$. We claim that $Q_{n-b_j}^{b_j}$ is contained in $Q_X$ and therefore $Q_X$ has corank at least $b_1$. Indeed, let $G_{b_j+1}$ be a general isotropic linear space of dimension $b_j+1$. Define the Schubert variety
$$\Sigma:=\{\Lambda\in OG(k,n)|\dim(\Lambda\cap G_{b_j+1})\geq j\}.$$
Specializing $X$ to a Schubert variety, it is easy to check $\sigma_{a;b}\cdot [\Sigma]\neq 0$. Therefore $X\cap \Sigma\neq \emptyset.$ The projective linear space $\mathbb{P}(G_{b_j+1})$ meets $\mathbb{P}(F_{b_j})^\perp$ in a projective point $p\in Q_{n-b_j}^{b_j}$, which must be contained in every $\Lambda\in X\cap\Sigma$. Since every general point in $Q_{n-b_j}^{b_j}$ can occur in this way, the quadric $Q_{n-b_j}^{b_j}$ is contained in $Q_X$. By the corank bound on the sub-quadrics, 
\begin{eqnarray}
\text{corank}(Q_X)&\geq& \text{corank}(Q_{n-b_j}^{b_j})-\left[\dim(Q_X)-\dim(Q_{n-b_j}^{b_j})\right]\nonumber\\
&=&b_j-[(n-b_1-2)-(n-b_j-2)]=b_1.\nonumber
\end{eqnarray}
Since $Q_X$ is contained in the smooth quadric hypersurface $Q$ in $\mathbb{P}^{n-1}$,
\begin{eqnarray}
\text{corank}(Q_X)&\leq& \text{corank}(Q)+\left[\dim(Q-\dim(Q_X)\right]\nonumber\\
&=&0+[(n-2)-(n-b_1-2)]=b_1.\nonumber
\end{eqnarray}
We conclude that corank$(Q_X)=b_1$. Let $\mathbb{P}(F_{b_1})$ be the singular locus of $Q_X$, then for every $\Lambda\in X$, $\mathbb{P}(\Lambda)\subset Q_X=\mathbb{P}(F_{b_1}^\perp)\cap Q$, and therefore $\Lambda\subset F_{b_1}^\perp$ as desired. It is clear that such $F_{b_1}$ is unique and is contained in $F_{b_j}$.

Now assume $j'\geq 2$. Take a general linear space $G_{b_1+1}$ of dimension $b_1+1$. The projective linear space $\mathbb{P}(G_{b_1+1})$ intersect $\mathbb{P}(F_{b_1}^\perp)$ in a projective point $\mathbb{P}(P_1)=p\in Q_X$. Define
$$X_p:=\{\Lambda\in X|p\in\mathbb{P}(\Lambda)\}.$$
$X_p$ is the intersection of $X$ with the Schubert variety $\Sigma:=\{\Lambda\in OG(k,n)|\dim(\Lambda\cap G_{b_1+1})\geq 1\}$. Specializing $X$ to a Schubert variety, we can see that the class of $X_p$ is given by $\sigma_{\bar{a};\bar{b}}$, where $\bar{a}_1=1$, $\bar{a}_i=a_{i-1}+1$ for $i\geq 2$ and $a_{i-1}\leq b_1$, $\bar{a}_i=a_{i-1}$ for $i\geq 2$ and $a_{i-1}>b_1$, $\bar{b}_t=b_{t+1}$ for $1\leq t\leq k-s-1$. Set $\bar{F}_{b_j}=$span$(P_1,P_1^\perp\cap F_{b_j})$. For every $\Lambda\in X_p$, $\Lambda\subset P_1^\perp$, and therefore
\begin{eqnarray}
\Lambda\cap F_{b_j}^\perp&\subset& (\Lambda\cap P_1^\perp)\cap F_{b_j}^\perp\nonumber\\
&\subset&\Lambda\cap P_1^\perp\cap (P_1^\perp\cap F_{b_j})^\perp\nonumber\\
&\subset&\Lambda\cap \bar{F}_{b_j}^\perp.\nonumber
\end{eqnarray}
Notice that for a general $p\in Q_X$, $p\notin\mathbb{P}(F_{b_j})^\perp$ if $j>1$. Therefore for every $\Lambda\in X_p$,
$$\dim(\Lambda\cap \bar{F}_{b_j}^\perp)\geq 1+\dim(\Lambda\cap F_{b_j}^\perp)\geq k-j+2.$$
By induction, there exists a unique isotropic subspace $\bar{F}^p_{b_{j'}}\subset \bar{F}_{b_j}$ of dimension $b_{j'}=\bar{b}_{j'-1}$ such that 
$$\dim(\Lambda\cap (\bar{F}^p_{b_{j'}})^\perp)\geq k-j'+2,\forall \Lambda\in X_p.$$
Varying $p\in Q_X$, we get a set $Z^\circ$ of isotropic subspaces of dimension $b_{j'}$. Let $Z$ be the Zariski closure of $Z^\circ$ in $OG(b_{j'},n)$. We claim that the class of $Z$ is given by $\sigma_{1,...,b_1,b_1+2,...,b_{j'};b_1}$. To see this, let $I$ be the Zariski closure of the locus of all the possible pairs $(p,\bar{F}_{b_{j'}}^p).$ Let $\pi_1:I\rightarrow Q_X$ and $\pi_2:I\rightarrow OG(b_{j'},n)$ be the two projections. Then $\dim(\pi_1(I))=\dim(Q_X)=n-b_1-2$ and a general fiber of $\pi_1$ has dimension $0$. Therefore $\dim(I)=n-b_1-2$. For every $q\in \mathbb{P}(\bar{F}^p_{b_{j'}})\subset \mathbb{P}((\bar{F}^p_{b_j})^\perp)$ and $r \in \bar{F}_{b_{j'}}^\perp$, there exists $\Lambda\in X_p$ that contains the projective line $\overline{pq}$. Therefore a general fiber of $\pi_2$ has dimension $b_{j'}-1$. Thus
$$\dim(Z)=\dim(\pi_2(I))=n-b_1-b_{j'}-1.$$
Now we intersect $Z$ with the Schubert varieties of complementary dimension in $OG(b_{j'},n)$. Let $H_{b_1+1}$ be a general isotropic subspace of dimension $b_1+1$. Define the Schubert variety $$\Sigma=\Sigma_{b_1+1;b_{j'}-1,...,b_1+1,b_1-1,...,0}:=\{W\in OG(b_{j'},n)|\dim(W\cap H_{b_1+1})\geq 1\}.$$
$\mathbb{P}(H_{b_1+1})$ meet $Q_X$ in a projective point $p\in Q_X$, which gives a unique $\bar{F}_{b_{j'}}^p\in Z\cap \Sigma$. On the other hand, suppose for a contradiction that $[Z]$ has another summand $\sigma_{\mu;\nu}$ such that $\mu_{b_{j'}-1}\geq b_{j'}+1$. Let $H_{b_{j'}}$ be a general isotropic subspace of dimension $b_{j'}$. Let $\Sigma'$ be the Schubert variety defined by
$$\Sigma':=\{W\in OG(b_{j'},n)|\dim(W\cap H_{b_{j'}}^\perp)\geq 2\}.$$ 
By assumption, $[\Sigma']\cdot \sigma_{\mu;\nu}\neq 0$ and therefore $\Sigma'\cap Z\neq \emptyset$. Hence there exists $p\in Q_X$ such that $\dim(\bar{F}_{b_{j'}}^p\cap H_{b_{j'}}^\perp)\geq 2$. Since $\bar{F}_{b_{j'}}^p=\bar{F}_{b_{j'}}^q$ for $q\in \bar{F}_{b_{j'}}^p$, we can assume $p\in H_{b_{j'}}^\perp$. Let $p'\neq p$ be another point in $ \bar{F}_{b_{j'}}^p\cap H_{b_{j'}}^\perp$. Then there exists $\Lambda\in X_p$ such that $\Lambda$ contains both $p$ and $q$, and hence $\dim(\Lambda\cap H_{b_{j'}}^\perp)\geq 2$. However, it contradicts the fact that $\sigma_{a;b}\cdot[\Sigma']=0$. We conclude that $[Z]=\sigma_{1,...,b_1,b_1+2,...,b_{j'};b_1}$.
 
By \cite[Corollary 5.9]{YL}, the sub-index $b_{j'}$ is rigid with respect to the Schubert class $\sigma_{1,...,b_1,b_1+2,...,b_{j'};b_1}$. Therefore there exists an isotropic subspace $F_{b_{j'}}$ such that $\dim(F_{b_{j'}}\cap \bar{F}_{b_{j'}})\geq b_{j'}-1$ for all $\bar{F}_{b_{j'}}\in Z$. Now it is easy to check for every $\Lambda\in X$, $\Lambda\in X_p$ for some $p\in Q_X$, and hence
\begin{eqnarray}
\dim(\Lambda\cap F_{b_{j'}}^\perp)&=&\dim(\Lambda)-\dim(F_{b_{j'}})+\dim(\Lambda^\perp\cap F_{b_{j'}})\nonumber\\
&\geq&\dim(\Lambda)-\dim(F_{b_{j'}})+\dim(\Lambda^\perp\cap \bar{F}_{b_{j'}}^p)-1\nonumber\\
&=&\dim(\Lambda\cap (\bar{F}_{b_{j'}}^p)^\perp)-1\nonumber\\
&\geq &k-j'+1.\nonumber
\end{eqnarray}
To show that $F_{b_{j'}}$ is contained in $F_{b_j}$, suppose for a contradiction that there exists a point $q\in\mathbb{P}(F_{b_j'})$ that is not contained in $\mathbb{P}(F_{b_j})$. Applying the previous argument to $b_{j'}$ and $b_1$, we obtain $F_{b_1}\subset F_{b_{j'}}$. Let $M$ be the span of $q$ and $F_{b_1}$. Since $b_{j'}$ is essential, $M\subsetneq F_{b_{j'}}$. Take a general $p$ in $M^\perp$ that does not contained in the span of $q$ and $F_{b_j}$. By induction, we have $$\text{span}(p,p^\perp\cap F_{b_{j'}})\subset \text{span}(p,p^\perp\cap F_{b_j}).$$
On the other hand, by the construction, the point $q$ is contained in the left side, but not contained in the right side. We reach a contradiction. We conclude by induction that the statement is true for all essential $1\leq j'<j$.
\end{proof}

\begin{Rem}\label{remark4.17}
The case when $n$ is even and $b_j=\frac{n}{2}-1$ can be done similarly due to the observation that if a sub-quadric of $Q$ contains a maximal isotropic subspace, then it must have the maximal possible corank.
\end{Rem}
\begin{Ex}
Consider the Schubert class $\sigma_{3^2;3^1,1^1,0^2}$ in $OF(2,4;11)$. Under the first projection, the sub-index $b_2=1$ is not rigid with respect to the Schubert class $\sigma_{;3,1}$ in $OG(2,11)$. Under the second projection, $b_2=1$ is not essential. Nevertheless, we claim that the sub-index $b_2$ is rigid with respect to the class $\sigma_{3^2;3^1,1^1,0^2}$. 

Let $X$ be a representative of $\sigma_{3^2;3^1,1^1,0^2}$ in $OF(2,4;11)$. By Theorem \ref{rigid in og}, the Schubert class $\sigma_{3;3,1,0}$ is rigid. Therefore there exists a unique isotropic subspace $F_3$ of dimension $3$ such that for all $(\Lambda_1,\Lambda_2)\in X$, $\dim(F_3\cap \Lambda_2)\geq 1$ and $\dim(F_3^\perp\cap \Lambda_2)\geq 2$. By Theorem \ref{rigidindexinoff}, $\dim(F_3^\perp\cap \Lambda_1)\geq 1$. And by Lemma \ref{rankbound}, there exists an isotropic subspace $F_1$ of dimension $1$ such that $\dim(F_1^\perp\cap \Lambda_1)\geq 2$. 
\end{Ex}

These observations lead to the following definition:
\begin{Def}
Let $\sigma_{a^\alpha;b^\beta}$ be a Schubert class in $OF(d_1,...,d_k;n)$. Set
 $$E:=\{a_i|a_i \text{ essential}\}\cup\{b_j|b_j\text{ essential}\}.$$ We define a relation `$\Rightarrow$' between two sub-indices: $b_{j_1}\Rightarrow b_{j_2}$ if $j_2<j_1\neq\frac{n}{2}-1$ and $b_{j_2}$ is essential in $(\pi_t)_*(\sigma_{a^\alpha;b^\beta})$ for some $t\geq \min(\beta_{j_1},\beta_{j_2})$; $a_i\Rightarrow b_j$ and also $b_j\Rightarrow a_i$ if $a_i=b_j\neq \frac{n}{2}-1$. We then extend this relation to $E$ by transitivity and reflexivity (which we also denote by `$\Rightarrow$').
\end{Def}

We obtain the following characterization of rigid sub-indices:
\begin{Cor}\label{rigidindexinof}
Let $\sigma_{a^\alpha;b^\beta}$ be a Schubert class in $OF(d_1,...,d_k;n)$. An essential sub-index $a_i$ or $b_j$ is rigid if and only if there exists an element $e\in E$ such that $e\Rightarrow a_i$ (or $e\Rightarrow b_j$ resp.) and the sub-index $e$ is rigid with respect to the class $(\pi_t)_*(\sigma_{a^\alpha;b^\beta})$ for some $t$.
\end{Cor}
\begin{proof}
If all the conditions hold, then $a_i$ or $b_j$ is rigid following from Theorem \ref{rigidindexinoff}, Lemma \ref{atob} and Lemma \ref{rankbound}. If one of the conditions in the statement fails, then we construct the counter-examples that are not Schubert varieties.

First assume $a_i$ is not rigid under all the canonical projections. If $a_i\neq b_j$ for all $j$, then set $t=k$ if $a_{i+1}\neq a_i+1$ and $t=\alpha_{i+1}-1$ if $a_{i+1}=a_i+1$. Since $a_i$ is essential, $\alpha_i\leq t$. By assumption, $a_i$ is not rigid with respect to the class $(\pi_t)_*(\sigma_{a^\alpha;b^\beta})$. Using the construction in Remark \ref{exinof}, there exists a non-Schubert subvariety $Y$ in $OG(d_t,n)$ with a partial flag $$F_{a_1}\subset ...\subset \hat{F}_{a_i}...\subset F_{a_s}\subset F_{b_{k-s}}^\perp\subset...\subset F_{b_1}^\perp$$ such that $[Y]=(\pi_t)_*(\sigma_{a^\alpha;b^\beta})$ and $\dim(F_{a_{i'}}\cap \Lambda)\geq \mu_{i',t}$, $\forall i'\neq i$, $\alpha_{i'}\leq t$ and $\dim(F_{b_j}^\perp\cap \Lambda)\geq \nu_{j,t}$, $\forall \beta_j\leq t$. Define
$$X:=\{(\Lambda_1,...,\Lambda_k)\in OF(d_1,...,d_k)|\Lambda_t\in Y,\dim(F_{a_i'}\cap\Lambda_r)\geq \mu_{i',r}, \dim(F_{b_j}^\perp\cap \Lambda_r)\geq \nu_{j,r}, \forall i'\neq i,r\neq t.\}$$
Then $[X]=\sigma_{a^\alpha;b^\beta}$ while $X$ is not a Schubert variety.

If $a_i=b_j$ and there does not exist $e\in J$ such that $e\Rightarrow a_i$ and $e$ is rigid with respect to the class $(\pi_t)_*(\sigma_{a^\alpha;b^\beta})$ for some $t$, set $J:=\{j'|b_{j'}=a_{i'}\text{ for some }i'\geq i\}$. Take a sequence of isotropic subspaces and sub-quadrics
$$F_{a'_1}^{\alpha_1}\subset...\subset F_{a'_s}^{\alpha_s}\subset (Q_{n-b_{d_k-s}}^{b_{d_k-s}-1})^{\beta_{d_k-s}}\subset...\subset (Q_{n-b_j}^{b_{j}-1})^{\beta_j}\subset (Q_{n-b_{j-1}}^{{b_{j-1}}})^{\beta_{j-1}}\subset...\subset (Q_{n-b_1}^{b_1})^{\beta_1},$$
where $a'=(a'_1,...,a'_s)$ is obtained from $a=(a_1,...,a_s)$ by changing all the $a_{i'}$, $i'\geq i,$ to the sequence that consists of all the numbers from $b_{j}+1$ to $\left[\frac{n}{2}\right]$ excluding the numbers equal to $b_{j'}$, $j'\in J$). Let $X$ be the restriction variety defined by the above sequence. Then $X$ is a non-Schubert representative of $\sigma_{a^\alpha;b^\beta}$.

Now assume there does not exist $e\in J$ such that $e\Rightarrow b_j$ and $e$ is rigid with respect to the class $(\pi_t)_*(\sigma_{a^\alpha;b^\beta})$ for some $t$. Set $t=k$ if $b_{j-1}\neq b_j-1$ and $t=\beta_{j-1}-1$ if $b_{j-1}=b_j-1$. Since $b_j$ is essential by assumption, $\beta_j\leq t$. Set $j_0:=\min\{J\}=\min\{j'|j'\geq j, b_{j'}=a_i\text{ for some }i\}$. Take a sequence 
$$F_{a'_1}^{\alpha_1}\subset...\subset F_{a'_s}^{\alpha_s}\subset (Q_{n-b_{d_k-s}}^{b_{d_k-s}-1})^{\beta_{d_k-s}}\subset...\subset (Q_{n-b_j}^{b_{j}-1})^{\beta_j}\subset (Q_{n-b_{j-1}}^{{b_{j-1}}})^{\beta_{j-1}}\subset...\subset (Q_{n-b_1}^{b_1})^{\beta_1},$$
where $a'=(a'_1,...,a'_s)$ is obtained from $a=(a_1,...,a_s)$ by changing all the $a_{i'}\geq b_{j_0}$ to the sequence that consists of all the numbers from $b_{j_0}+1$ to $\left[\frac{n}{2}\right]$ excluding the numbers equal to $b_{j'}$, $j'\in J$). For each $j'>j$ such that $j'\not\Rightarrow j$, replace $Q_{n-b_{j'}}^{b_{j'}-1}$ with another quadric $\bar{Q}_{n-b_{j'}}^{b_{j'}}$ that contains $F_{a_s}$ and is contained in $Q_{n-b_{j-1}}^{b_{j-1}}$. If $j''>j'$ and $j''\not\Rightarrow j$, then we require $\bar{Q}_{n-b_{j''}}^{b_{j''}}\subset\bar{Q}_{n-b_{j'}}^{b_{j'}}$. Let $X$ be the subvariety defined by the resulting sequence of isotropic subspaces and sub-quadrics. Then $X$ is a non-Schubert representative of $\sigma_{a^\alpha;b^\beta}$.
\end{proof}

\begin{Def}\label{compof}
Let $\sigma_{a^\alpha;b^\beta}$ be a Schubert class in $OF(d_1,...,d_k;n)$. Set
 $$E:=\{a_i|a_i \text{ essential}\}\cup\{b_j|b_j\text{ essential}\}.$$ We define a relation `$\rightarrow$' between two sub-indices: 
\begin{itemize}
\item $a_{i_1}\rightarrow a_{i_2}$ if $i_1\leq i_2$ and $a_{i_2}$ is essential in $(\pi_t)_*(\sigma_{a^\alpha;b^\beta})$ for some $t\geq \min(\alpha_{i_1},\alpha_{i_2})$. (If $n$ is even and $a_{i_2}=\frac{n}{2}$, then we do not require $a_{i_2}$ to be essential.) ; 
\item $a_i\rightarrow b_j$ if $a_i\leq b_{j}$;
\item $b_j\rightarrow a_i$ if $b_j\leq a_i$, $b_j\neq \frac{n}{2}-1$ and both $a_i$ and $b_j$ are essential in $(\pi_t)_*(\sigma_{a^\alpha;b^\beta})$ for some $t\geq \min(\alpha_i,\beta_{j})$; 
\item $b_{j_1}\rightarrow b_{j_2}$ if $j_1\leq j_2$ and $b_{j_1}$ is essential in $(\pi_t)_*(\sigma_{a^\alpha;b^\beta})$ for some $t\geq \min(\beta_{j_1},\beta_{j_2})$. 
\end{itemize}
We then extend this relation to $E$ by transitivity (which we also denote by `$\rightarrow$').
\end{Def}

\begin{Lem}\label{cbile}
Let $\sigma_{a;b}$ be a Schubert class in $OG(k;n)$, and let $X$ be a representative of $\sigma_{a;b}$. Assume there exist isotropic subspaces $F_{a_{i_1}}$, $F_{a_{i_2}}$, $F_{b_{j_1}}$, $F_{b_{j_2}}$ such that for all $\Lambda\in X$,
$$\dim(F_{a_{i_\gamma}}\cap \Lambda)\geq i_\gamma,\gamma=1,2;$$
$$\dim(F^\perp_{b_{j_\gamma}}\cap \Lambda)\geq k-j_\gamma+1,\gamma=1,2.$$
 (The sub-indices $a_{i_1},a_{i_2},b_{j_1},b_{j_2}$ need not to be rigid.)
\begin{itemize}
\item If $i_1\leq i_2$ and $a_{i_2}$ is essential, then $F_{a_{i_1}}\subset F_{a_{i_2}}$. (If $n$ is even and $a_{i_2}=\frac{n}{2}$, then we do not require $a_{i_2}$ to be essential);
\item If $j_1\leq j_2$ and $b_{j_1}$ is essential, then $F_{b_{j_1}}\subset F_{b_{j_2}}$;
\item If $a_{i_1}\leq b_{j_1}$, then $F_{a_{i_1}}\subset F_{b_{j_1}}$;
\item If $b_{j_1}\leq a_{i_1}$, $b_{j_1}\neq \frac{n}{2}-1$ and both $a_{i_1}$ and $b_{j_1}$ are essential then $F_{b_{j_1}}\subset F_{a_{i_1}}$.
\end{itemize}
\end{Lem}
\begin{proof}
For the first case, suppose for a contradiction that $F_{a_{i_1}}\not\subset F_{a_{i_2}}$. Take a one-dimensional isotropic subspace $G_1$ in $F_{a_{i_1}}$, which is not contained in $F_{a_{i_2}}$. Set $X_1:=\{\Lambda\in X|G_1\subset \Lambda\}$. Then $X_1$ has class $\sigma_{a';b'}$, where $b'=b$, $a'_1=1$, $a'_{u+1}=a_u+1$ for $u<i_1$, $a'_u=a_u$ for $u>i_1$. The condition $G_1\subset\Lambda$ implies $\Lambda\subset G_1^\perp$. Therefore
$$\dim(\Lambda\cap F_{a_{i_2}})=\dim(\Lambda\cap (G_1^\perp\cap F_{a_{i_2}})), \text{ for all }\Lambda\in X_1.$$
It is then necessary that $G_1^\perp\cap F_{a_{i_2}}=F_{a_{i_2}}$, i.e. $F_{a_{i_2}}\subset G_1^\perp$. Therefore the span $W$ of $G_1$ and $F_{a_{i_2}}$ is isotropic. Since $G_1$ is not contained in $F_{a_{i_2}}$, $\dim(W)=a_{i_2}+1$. If $F_{a_{i_2}}$ is a maximal isotropic subspace, then we already reach a contradition. If $F_{a_{i_2}}$ is not maximal, then for every $\Lambda\in X_1$, $\dim(\Lambda\cap W)\geq i_2+1$. Therefore $a_{i_2+1}=a'_{i_2+1}=a_{i_2}+1$. This contradicts the condition that $a_{i_2}$ is essential.

The second case follows from Lemma \ref{rankbound} and Remark \ref{remark4.17}.

For the third case, assume $a_{i_1}<b_{j_1}$. Using the first case, it is safe to replace $a_{i_1}$ with the largest essential sub-index less or equal to $b_{j_1}$. Similarly, using the second case, it is safe to replace $b_{j_1}$ by the smallest essential sub-index greater or equal to $a_{i_1}$. If $b_{j_1}=\frac{n}{2}-1$, then under an involution of quadrics, it reduces to the first case. If $a_{i_1}$ and $b_{j_1}$ are equal, then it follows from Lemma \ref{atob}. Now assume $a_{i_1}<b_{j_1}<\frac{n}{2}-1$ both are essential. Take a general point $p$ in $\mathbb{P}(F_{a_{i_1}})$. Define $X_p:=\{\Lambda \in X|p\in\mathbb{P}(\Lambda)\}.$ By specializing $X$ to a Schubert variety, the class of $X_p$ is given by $\sigma_{a';b'}$, where $a'_1=1$, $a'_{i+1}=a_{i}+1$ if $i<i_1$, $a'_{i}=a_{i}$ if $i>i_1$, $b'_{j}=b_{j}$ if $b_{j}\geq a_{i_1}$, $b'_{j}=b_{j}+1$ if $b_{j}<a_{i_1}$. In particular, $b'_{j_1}=b_{j_1}$. Suppose for a contradiction that $p\notin \mathbb{P}(F_{b_{j_1}})$. If $p\in \mathbb{P}(F_{b_{j_1}}^\perp)$, then $W:=$ span$(p,F_{b_{j_1}})$ is isotropic. For every $\Lambda\in X_p$, $\Lambda\subset p^\perp$, and therefore $\dim(\Lambda\cap F_{b_{j_1}}^\perp)=\dim( \Lambda\cap (p^\perp\cap F_{b_{j_1}}^\perp))=\dim(\Lambda\cap W^\perp)$. Since $W$ has dimension $b_{j_1}+1$, it contradicts the condition $b'_{j_1}=b_{j_1}$. If $p\notin \mathbb{P}(F_{b_{j_1}}^\perp)$, then let $U=p^\perp \cap F_{b_{j_1}}$. Since $p\in U^\perp$, for every $\Lambda\in X_p$, $\dim(\Lambda\cap U^\perp)\geq1+\dim(\Lambda\cap F_{b_{j_1}}^\perp)\geq k-j_1+2$. Since $\dim(U)=b_{j_1}-1$, $b'_{j_1-1}=b_{j_1}-1$, which implies $b'_{j_1}$ and therefore $b_{j_1}$ is not essential. We reach a contradiction.

For the last case, assume $b_{j_1}<a_{i_1}$ and both $a_{i_1}$ and $b_{j_1}$ are essential. Take a general point $p\in\mathbb{P}(F_{b_{j_1}})$. Define $X_p:=\{\Lambda\in X|p\in\mathbb{P}(\Lambda)\}$. Set $i_0=\min\{i|a_i\geq b_{j_1}\}$. Then $[X_p]=\sigma_{a'';b''}$, where $a''_1=1$, $a''_{i+1}=a_i+1$ if $i<i_0$, $a''_{i}=a_i$ if $i>i_0$, $b''_{j}=b_j$ if $j\geq j_1$, $b''_j=b_j+1$ if $j<j_1$. In particular, $a''_{i_1+1}=a_{i_1+1}$. Suppose for a contradiction that $p\notin\mathbb{P}(F_{a_{i_1}})$. Let $W'$ be the span of $p$ and $F_{a_{i_1}}$. Then for every $\Lambda\in X$, $\dim(\Lambda\cap W')\geq 1+\dim(\Lambda\cap F_{a_{i_1}})\geq i_1+1$. Since $W'$ has dimension $a_{i_1}+1$, $a''_{i_1+1}\leq a_{i_1}+1$. However, this implies $a_{i_1}$ is not essential, we reach a contradiction.
\end{proof}

As an application, we classify the rigid Schubert classes in $OF(d_1,...,d_k;n)$:
\begin{Thm}
Let $\sigma_{a^\alpha;b^\beta}$ be a Schubert class in $OF(d_1,...,d_k;n)$. The class $\sigma_{a^\alpha;b^\beta}$ is rigid if and only if all essential indices are rigid and the set $E$ defined in Definition \ref{compof} is totally ordered under the relation `$\rightarrow$'. (The sub-indices $a_{i_1},a_{i_2},b_{j_1},b_{j_2}$ need not to be rigid.)
\end{Thm}

\begin{proof}
If all essential sub-indices are rigid, then by Theorem \ref{rigidindexinoff}, for every representative $X$, there exists an isotropic subspace for each essential sub-index. If furthermore the set $E$ of all essential sub-indices is totally ordered, then by Lemma \ref{cbile}, these isotropic subspaces form a partial flag, and $X$ is the Schubert variety defined by this partial flag.

Conversely, if one of essential indices is not rigid, then we have constructed in the proof of Corollary \ref{rigidindexinof} a non-Schubert subvariety that represents the Schubert class $\sigma_{a^\alpha;b^\beta}$.

Now assume all essential sub-indices are rigid, but $E$ is not totally ordered under the relation `$\rightarrow$'. Take a complete flag of isotropic subspaces
$$F_1\subset...\subset F_{[n/2]},$$
and remove all the flag elements that does not correspond to an element in $E$. Since $E$ is not totally ordered, we can pick one place where the link is broke, say $a_{i_1}$ and $a_{i_2}$ are essential, $a_{i_1}\not\rightarrow a_{i_2}$, and there is no other essential sub-index between $a_{i_1}$ and $a_{i_2}$. Since $a_{i_2}$ is necessary not essential in the $k$-th projection (otherwise we must have $a_{i_1}\rightarrow a_{i_2})$, $a_{i_2+1}=a_{i_2}+1$, and hence there is not $b_j$ such that $b_j=a_{i_2}$. Replace $F_{a_{i_2}}$ with another isotropic subspace $G_{a_{i_2}}$ of the same dimension such that $$F_{a_{i_1}-1}\subset G_{a_{i_2}}\subset F_{a_{i_2}+1},$$
and $\dim(F_{a_{i_1}}\cap G_{a_{i_2}})=a_{i_1}-1$.

Define 
\begin{eqnarray}
& &X:=\{\Lambda\in OF(d_1,...,d_k)|\dim(G_{a_{i_2}}\cap \Lambda_t)\geq \mu_{i_2,t},\dim(F_{a_i}\cap \Lambda_t)\geq \mu_{i,t}, \nonumber\\
& &\ \ \ \ \ \ \ \ \ \ \ \ \ \  \ \ \ \ \ \ \ \ \ \ \ \ \ \ \ \ \ \ \ \ \dim(F_{b_j}^\perp\cap \Lambda_t)\geq\nu_{j,t}, \text{ for }a_i,b_j\in E-\{i_2\},1\leq t\leq k\}\nonumber
\end{eqnarray}
Then $X$ represents the Schubert class $\sigma_{a^\alpha}$ but is not a Schubert variety.

The case of $b_{j}\not\rightarrow a_i$ and the case of $b_{j_1}\not\rightarrow b_{j_2}$ can be constructed similarly.
\end{proof}

\section{Generalized rigidity problem}\label{multi}
In this section, we discuss a similar question but much more generalized than the regular rigidity problem. We will focus on the case of Grassmannians, even though a similar question could be formulated for partial flag varieties. As an application, we prove certain Schubert classes in orthogonal Grassmannians are multi-rigid. 

\subsection{Generalized rigidity problem in Grassmannians}
Let $X$ be an irreducible subvariety in the Grassmannian $G(k,n)$ with $[X]=\sum c_a\sigma_{a}$, where $\sigma_a$ are Schubert classes and $c_a$ are non-negative integers. For each $1\leq i\leq k$, set $$\gamma_i(X):=\min\{d|\exists F_d\in G(d,n)\text{ s.t. }\dim(F_d\cap \Lambda)\geq i, \forall \Lambda\in X\}.$$
We ask the following question:
\begin{Prob}
Let $X$ be an irreducible subvariety with class $[X]=\sum c_a\sigma_{a}$. For each $1\leq i\leq k$, when does $\gamma_i(X)$ only depend on the class $[X]$?
\end{Prob}
If $X=\Sigma_a$ is a Schubert variety, then we have $\gamma_i(X)=a_i$. If furthermore $a_i$ is rigid, then $\gamma_i(X)$ is preserved under rational equivalence. Recall also that a Schubert class is called {\em multi-rigid} if every multiple of it can only be represented by a union of Schubert varieties. More generally,
\begin{Def}
Let $\sigma_a$ be a Schubert class in the Grassmannian $G(k,n)$. An essential sub-index $a_i$ is called {\em multi-rigid} if for every irreducible representative $X$ of the class $m\sigma_a$, $m\in\mathbb{Z}^+$, there exists a unique subspace $F_{a_i}$ of dimension $a_i$ such that for every $k$-plane $\Lambda$ parametrized by $X$,
$$\dim(\Lambda\cap F_{a_i})\geq i.$$
\end{Def}
It is clear that if $[X]=m\sigma_a$ and $a_i$ is multi-rigid, then $\gamma_i(X)=a_i$. In \cite{YL2}, we classified the multi-rigid sub-indices in Grassmannians:
\begin{Thm}\cite[Theorem 4.1.5]{YL2}\label{multirigid in g}
Let $\sigma_a$ be a Schubert class in $G(k,n)$. Set $a_0=0$ and $a_{n+1}=\infty$. An essential sub-index $a_i$ is multi-rigid if and only if 
\begin{enumerate}
\item either $a_{i-1}+1=a_i\leq a_{i+1}-3$; or
\item $a_i=n$.
\end{enumerate}
\end{Thm}

The following proposition shows that $\gamma_k(X)$ is only dependent on the class of $X$:

\begin{Prop}\label{generalmulti}
Let $X$ be an irreducible subvariety of $G(k,n)$ with class $[X]=\sum_{a\in A}m_a\sigma_a$, where $m_a\in\mathbb{Z}^+$ and $A$ is a finite set of Schubert indices. Assume 
$$\max\{a_{k-1}|a\in A\}+1=\max\{a_k|a\in A\}.$$
Then $\gamma_k(X)=\max\{a_k|a\in A\}$.
\end{Prop}
\begin{proof}
Let $Y$ be the subvariety swept out by $\mathbb{P}^{k-1}$ parametrized by $X$. Write $c=\max\{a_k|a\in A\}$. We claim that $Y$ is a projective linear space of projective dimension $c-1$.

Let $G_{n-c}$ and $G_{n-c+1}$ be general linear spaces of (affine) dimension $n-c$ and $n-c+1$ resp. Consider the Schubert varieties
$$\Sigma_{n-c,n-k+2,...,n}:=\{\Lambda\in G(k,n)|\dim(\Lambda\cap G_{n-c})\geq 1\}$$
and
$$\Sigma_{n-c+1,n-k+2,...,n}:=\{\Lambda\in G(k,n)|\dim(\Lambda\cap G_{n-c+1})\geq 1\}.$$
Then we can compute the intersection product using Pieri's Lemma, 
$$[X]\cdot\sigma_{n-c,n-k+2,...,n}=0,$$
$$[X]\cdot\sigma_{n-c+1,n-k+2,...,n}=\sum_{a\in A'}m_a\sigma_{a'}\neq 0,$$
where $A':=\{a\in A|a_{k}=c\}$ and $a'$ is obtained from $a$ by the relation $a'_1=1$, $a'_{i+1}=a_i+1$, $1\leq i\leq k-1$. This implies 
$$Y\cap\mathbb{P}(G_{n-c})=\emptyset\text{ and }Y\cap\mathbb{P}(G_{n-c+1})\neq \emptyset.$$
Therefore $Y$ has a projective dimension $c-1$.

Now we need to show $Y$ is linear. Let $p$ be a general point in $Y$. Let $X_p:=\{\Lambda\in X|p\in\mathbb{P}(\Lambda)\}$. Then $X_p$ has class $[X_p]=\sum_{a\in A'}\frac{m_a}{m_p}\sigma_{a'}$, where $m_p|gcd\{m_a|a\in A'\}$. Let $Z$ be the variety swept out by $\mathbb{P}^{k-1}$ parametrized by $X_p$. Since $\max\{a_{k-1}|a\in A\}+1=\max\{a_k|a\in A\},$ a similar argument above shows that $\dim(Z)=c-1=\dim(Y)$. Since $X$ is irreducible, so is $Y$. Therefore $Y=Z$. Notice that $Z$, and therefore $Y$, is contained in the tangent space $T_pY$ at the point $p$. This force $Y$ being a linear space. This completes the proof.
\end{proof}

\begin{Cor}\label{1 rigid in sum}
Let $X$ be an irreducible subvariety of $G(k,n)$ with class $[X]=\sum_{a\in A}m_a\sigma_a$, where $m_a\in\mathbb{Z}^+$ and $A$ is a finite set of Schubert indices. For each $a\in A$, set $d_a:=\max\{i|a_i=i\}$. If $a_1\geq 2$, set $d_a=0$. Let $d=\min\{d_a|a\in A\}$. Assume $d\geq 1$, and there exists $a\in A$ such that $a_{d}=d$ and $a_{d+1}\geq d+3$. Then there exists a unique linear space $F_{d}$ of dimension $d$ that is contained in all $k$-planes parametrized by $X$.
\end{Cor}
\begin{proof}
Follows from the duality $G(k,n)\cong G(n-k,n)$ and Proposition \ref{generalmulti}.
\end{proof}
The following result generalizes Proposition \ref{generalmulti}:
\begin{Thm}\label{rigid index in mix}
Let $X$ be an irreducible subvariety with class $[X]=\sum_{a\in A}c_a\sigma_a$, $c_a\in \mathbb{Z}^+$. Set $m_i=\max\limits_{a\in A}\{a_i\}$. Define $$A_1:=\{a\in A|a_1=m_1\}$$ and inductively
$$A_i:=\{a\in A_{i-1}|a_i=\max\limits_{a'\in A_{i-1}}\{a_i'\}\},\  2\leq i\leq k.$$
If $\max\limits_{a\in A_i}\{a_i\}=m_i$ and there exists an index $a\in A_i$ such that $a_{i-1}+1=m_i\leq a_{i+1}-3$, then $\gamma_i(X)=m_i$.
\end{Thm}

\begin{proof}
We use induction on $m_i$. If $m_i=i$, then it reduces to Corollary \ref{1 rigid in sum}. 

Assume $m_i>i$. Consider the incidence correspondence
$$I=\{(\Lambda,H)|\Lambda\in X,\Lambda\subset H, H\in (\mathbb{P}^{n-1})^*\}$$
Let $\pi_1:I\rightarrow X$ and $\pi_2:I\rightarrow (\mathbb{P}^{n-1})^*$ be the two projections. Let $\dim(X)=d$. The fibers of $\pi_1$ have dimension $n-k-1$ and therefore $I$ has dimension $d+n-k+1$.

Let $H$ be a general point in $\pi_2(I)$. Set $i_0:=\{i'|m_{i'}>i'\}$ and $B:=\{a\in A|a_{i_0}>i_0\}$. Then the fiber over $H$ has class $\sum\limits_{a\in B}c_a\sigma_{a'}$ where $a'_i=a_i$ if $a_i=i$ and $a'_i=a_i-1$ if $a_i\geq2$. By induction, there exists a unique linear subspace $F^H_{m_i-1}$ for each irreducible component of the fiber over $H$. Varying $H$ in $\pi_2(I)$ and let $Y$ be the resulting locus in $G(m_{i-1},n)$. We claim that $[Y]=c\sigma_b$ where $b_{m_{i}-1}=m_i$ and $b_{m_i-2}=m_i-1$.

Let $Z$ be Zariski closure of the locus of all possible pairs $(F_{m_i}^H,H)$ in $G(m_i-1,n)\times (\mathbb{P}^{n-1})^*$. Let $\pi_3:Z\rightarrow G(m_i-1,n)$ and $\pi_4:Z\rightarrow (\mathbb{P}^{n-1})^*$ be the two projections. Then $Y=\pi_3(Z)$. Let $s$ be a sub-index such that $a_s=s$ and $a_{s+1}\neq s+1$. Then we have $\dim(\pi_4(Z))=n-1-s$. Let $H$ be a general point in $\pi_4(Z)$. Then the fiber over $H$ has dimension $0$, and therefore $\dim(Z)=n-1-s$. Notice that the fibers of $\pi_3$ have dimension at most $n-m_i$, and therefore
$$\dim(\pi_3(Z))\geq m_i-s-1.$$
By assumption, there exists an index $a\in A_i$ such that $m_i=a_{i-1}+1\neq i$, and thus
$$m_i-s-1\geq 2.$$
Therefore $[Y]=c\sigma_b$ where $b_{m_i-1}=b_{m_i-2}+1=m_i$. Apply Theorem \ref{multirigid in g} to each irreducible component of $Y$, we obtain a finite number of linear subspaces $F_{m_i}^1$, ..., $F_{m_i}^t$ of dimension $m_i$ such that for every $F_{m_{i-1}}\in Y$, $F_{m_i-1}\subset F_{m_i}^{j}$ for some $1\leq j\leq t$. Now it is easy to check that for every $\Lambda\in X$, $\dim(\Lambda\cap F_{m_i}^{j})\geq j$ for some $j$. Set $X_{j}:=\{\Lambda\in X|\dim(\Lambda\cap F_{m_i}^j)\geq i\}$. Then $X=X_1\cup...\cup X_t$. Since $X$ is irreducible, $X=X_j$ for some $j$. Therefore $\mu_i(X)=m_i$.

\end{proof}

\subsection{Multi-rigidity problem in orthogonal Grassmannians}
As an application of the results in the previous section, we investigate the multi-rigidity problem in the orthogonal Grassmannian $OG(k,n)$, which can be viewed as a subvariety in $G(k,n)$ via the natural inclusion
$$OG(k,n)\hookrightarrow G(k,n).$$

\begin{Def}
Let $\sigma_{a;b}$ be a Schubert class for $OG(k,n)$. An essential sub-index $a_i$ is called {\em multi-rigid} if for every irreducible representative $X$ of the class $m\sigma_{a;b}$, $m\in\mathbb{Z}^+$, there exists an isotropic subspace $F_{a_i}$ of dimension $a_i$ such that for every $k$-plane $\Lambda$ parametrized by $X$,
$$\dim(\Lambda\cap F_{a_i})\geq i.$$
Similarly, an essential sub-index $b_j$ is called {\em multi-rigid} if for every irreducible representative $X$ of the class $m\sigma_{a;b}$, $m\in\mathbb{Z}^+$, there exists an isotropic subspace $F_{b_j}$ of dimension $b_j$ such that for every $k$-plane $\Lambda$ parametrized by $X$, $$\dim(\Lambda\cap F_{b_j}^\perp)\geq k-j+1.$$
\end{Def}

\begin{Ex}
Consider the Schubert class $\sigma_{1;1}$ in $OG(2,7)$. We claim that the class $\sigma_{1;1}$ is multi-rigid. Let $X$ be an irreducible representative of $m\sigma_{1;1}$, $m\in\mathbb{Z}^+$. There is a natural inclusion morphism
$$i:OG(2,7)\hookrightarrow G(2,7).$$
The push-forward class $i_*(m\sigma_{1;1})$ is the class $2m\sigma_{1,5}\in A(G(2,7))$. By Theorem \ref{multirigid in g}, the sub-index $1$ is multi-rigid with respect to the Schubert class $\sigma_{1,5}$. Since $X$ is irreducible, there exists a unique affine line $F_1$ such that 
$$F_1\subset \Lambda,\text{ for all }\Lambda\in X.$$
This shows the sub-index $1$ is multi-rigid with respect to the Schubert class $\sigma_{1;1}$. Moreover, observe that
$$F_1\subset\Lambda\Rightarrow \Lambda\subset F_1^\perp,$$
and hence $X$ is the Schubert variety defined by
$$F_1\subset F_1^\perp.$$
Therefore the class $\sigma_{1;1}$ is multi-rigid.
\end{Ex}

As shown in the above example, we look at the push-forward of Schubert classes under the inclusion $i:OG(k,n)\hookrightarrow G(k,n)$. Consider the Schubert variety $\Sigma_{a;b}$ defined by the following sequence of isotropic subspaces and sub-quadrics:
$$L_\bullet:F_{a_1}\subset...\subset F_{a_s}\subset Q_{n-b_{k-s}}^{b_{k-s}}\subset...\subset Q_{n-b_1}^{b_1},$$
where $Q_{n-{b_j}}^{b_j}=Q\cap F_{b_j}^\perp$ are sub-quadrics of $Q$ of corank $b_j$. To compute its class in $G(k,n)$, we increase the corank of the leftmost sub-quardric by one at a time by specialization. If the difference between the lower and upper indices becomes 2, say $Q_{n-b_{k-s}}^{n-b_{k-s}-2}$, then this sub-quadric breaks into a union of two linear subspaces of dimension $n-b_{k-s}-1$. If the corank becomes 1 less than some $a_i$, say $Q_{n-b_{k-s}}^{a_i-1}$, $a_i\geq b_{k-s}+2$, then by the observation in Remark \ref{difference1}, the sequence splits into two sequences $L_1$ and $L_2$, where $L_1$ is obtained by replacing $Q_{n-b_{k-s}}^{a_i-1}$ with $Q_{n-b_{k-s}-1}^{a_i}$, and $L_2$ is obtained by replacing $F_{a_i}$ with $F_{a_i-1}$. Notice that the dimension of the isotropic flag elements does not change in $L_1$. 

Let $F_{a_i-1}$ be the singular locus of $Q_{n-b_{k-s}}^{a_i-1}$. Let $\Lambda$ be a general point of $X$ and set $W:=\Lambda\cap Q_{n-b_{k-s}}^{a_i-1}$. Then $W$ is an isotropic subspace of dimension $s+1$ contained in the quadric $Q_{n-b_{k-s}}^{a_i-1}$. Therefore the dimension of the intersection of $W$ and the singular locus $F_{a_i-1}$ is at least $s+1-\left[\frac{n-b_{k-s}-(a_i-1)}{2}\right]$. Thus $L_1$ is effective if and only if 
$$i-1\geq s+1-\left[\frac{n-b_{k-s}-(a_i-1)}{2}\right].$$

This observation gives the following lemma:
\begin{Lem}\label{maxindex}
Let $\sigma_{a;b}$ be a Schubert class in $OG(k,n)$, and $i:OG(k,n)\hookrightarrow G(k,n)$ be the natural inclusion. For each $1\leq j\leq k-s$, set $x_j:=\{i|a_i\leq b_j\}$. Let $a_i$ be an essential sub-index. Assume for all $j$ such that $b_j<a_i$, 
$$x_j\geq k-j+1-\left[\frac{n-b_j-(a_{x_j+1}-1)}{2}\right].$$
Then $$i_*(\sigma_{a;b})=2^{k-s}\sigma_{\bar{a}}+\sum_\lambda c_\lambda\sigma_\lambda,$$ where $\bar{a}_t=a_t$ and $\lambda_t\leq a_t$, for $1\leq t\leq i$.
\end{Lem}

\begin{Thm}\label{multiindex}
Let $\sigma_{a;b}$ be a Schubert class in $OG(k,n)$. Let $a_i$ be an essential sub-index. Assume for all $j$ such that $b_j<a_{i+1}$, $$x_j\geq k-j+1-\left[\frac{n-b_j-(a_{x_j+1}-1)}{2}\right].$$
Then $a_i$ is multi-rigid if one of the following conditions hold:
\begin{itemize}
\item $i<s$ and $a_{i-1}+1=a_i\leq a_{i+1}-3$;
\item $i=s$, $a_s=a_{s-1}+1$ and $a_s\leq n-b_{k-s}-(s-x_j)-4$.
\end{itemize}
\end{Thm}
\begin{proof}
Let $X$ be a representative of $m\sigma_{a;b}$, $m\in\mathbb{Z}^+$. Let $i:OG(k,n)\hookrightarrow G(k,n)$ be the natural inclusion. By Lemma \ref{maxindex}, $i(X)$ has class  $2^{k-s}m\sigma_{\bar{a}}+\sum_\lambda mc_\lambda\sigma_\lambda$, where $\bar{a}_t=a_t$ and $\sigma_\lambda$ are Schubert classes such that $\lambda_t\leq \bar{a}_t$, $1\leq t\leq i+1$. By Theorem \ref{rigid index in mix}, $\mu_i(X)=a_i$. Therefore there exists a linear subspace $F_{a_i}$ such that $\dim(\Lambda\cap F_{a_i})\geq i$ for all $\Lambda\in X$. Let $G_{a_i-1}$ be a general isotropic subspace of dimension $a_i-1$, and let $\Sigma:=\{\Lambda\in OG(k,n)|\dim(\Lambda\cap G_{a_i-1}^\perp)\geq k-i+1\}$ be the corresponding Schubert variety in $OG(k,n)$. If $F_{a_i}$ is not isotropic, then $\mathbb{P}(F_{a_i})$ meets $\mathbb{P}(G_{a_i-1})^\perp$ in a projective point outside $Q$, and therefore $X\cap \Sigma=\emptyset$. However, $m\sigma_{a;b}\cdot[\Sigma]\neq 0$, we reach a contradiction. Therefore $F_{a_i}$ is an isotropic subspace. 
\end{proof}

\begin{Rem}
If $a_{i+1}\neq a_i+1$ and either $a_{i+1}-a_i=2$ or $a_i-a_{i-1}\geq 2$, and furthermore this is no $j$ such that $a_i=b_j$, then it is possible to construct an irreducible subvariety that represents $m\sigma_{a;b}$ for every $m\in \mathbb{Z}^+$. We give a brief construction of such subvarieties, using the idea in \cite{Coskun2013} and \cite{Coskun2014RigidityOS}.

We first construct an irreducible subvariety in $G(s,a_s)$ (or in $G(s,a_s+1)$ if $i=s$) that represents the class $m\sigma_{a_\bullet}$. If $a_{i}-a_{i-1}\geq 2$, take a partial flag 
$$F_{a_1}\subset...\subset F_{a_{i-1}}\subset F_{a_i-2}\subset F_{a_i+1}\subset F_{a_{i+1}}\subset...\subset F_{a_s}.$$
Take a $\mathbb{P}^2$ in $\mathbb{P}(F_{a_i+1})$ that doesn't meet $\mathbb{P}(F_{a_i-2})$, and take a smooth plane curve $C$ of degree $m$ in it. Let $Y$ be the cone over $C$ with vertex $\mathbb{P}(F_{a_i-2})$. (If $a_i=2$, then take $Y=C$.) Define
$$Z:=\{\Lambda\in G(s,a_s)|\dim(\Lambda\cap F_{a_j})\geq j\text{ for }j\neq i,\mathbb{P}(\Lambda)\cap Y\text{ is linear, }\dim(\mathbb{P}(\Lambda)\cap Y)\geq i-1\}.$$
Then $Y$ is an irreducible subvariety in $G(s,a_s)$ that represents the class $m\sigma_{a\bullet}$. If $a_{i+1}-a_i=2$, then using the duality $G(s,a_s)\cong G(a_s-s,a_s)$, the class $m\sigma_{a}$ is taken to the class $m\sigma_{a'}$, where $a'_{i'}-a'_{i'-1}\geq 2$ for some essential $a'_{i'}$ with respect to the Schubert index $a'$. The previous construction gives an irreducible subvariety $Z'$ that represents the class $m\sigma_{a'}$ in $G(a_s-s,a_s)$. Let $Z$ be the image of $Z'$ under the duality. Then $Z$ is irreducible and represents the class $m\sigma_{a}$ in $G(s,a_s)$.

Now using the subvariety $Z$ we can construct an irreducible subvariety $X$ in $OG(k,n)$ that represents the class $m\sigma_{a;b}$. If either $i\neq s$ or $a_s\neq \left[\frac{n}{2}\right]$, then take a maximal isotropic subspace $F_{\left[n/2\right]}$. The previous construction gives an irreducible subvariety $Z$ in $G(s,F_{\left[n/2\right]})$ represents the class $m\sigma_a$. Let $X$ be the Zariski closure of the following locus
$$\{\Lambda\in OG(k,n)|(\Lambda\cap F_{\left[n/2\right]})\in Z,\dim(\Lambda\cap F_{b_j}^\perp)=k-j+1\}.$$
Then $X$ is an irreducible representative of $m\sigma_{a;b}$. If $a_i=\left[\frac{n}{2}\right]$, then the same description will work except we replace $F_{\left[n/2\right]}$ with $Q_{a_i+1}^{a_i-2}$.

\end{Rem}

\begin{Prop}\label{bmulti}
Let $\sigma_{a;b}$ be a Schubert class in $OG(k,n)$. If $b_j=a_i$ and $a_i$ satisfies the condition in Theorem \ref{multiindex}, then $b_j$ is also multi-rigid.
\end{Prop}
\begin{proof}
Let $X$ be an irreducible representative of $m\sigma_{a;b}$, $m\in\mathbb{Z}^+$. Let $Y$ be the variety swept out by $\mathbb{P}^{k-1}$ parametrized by $X$. Consider the incidence correspondence
$$I:=\{(p,\mathbb{P}^{k-1})|p\in\mathbb{P}^{k-1},\mathbb{P}^{k-1}\in X\}.$$
Then $Y$ is the image of $I$ under the first projection. Under the second projection, $I$ maps onto $X$. Since $X$ is irreducible and all fibers of the second projection have dimension $k-1$ and are irreducible, by the fiber dimension theorem, $I$ is irreducible. Being the image of an irreducible variety, $Y$ is also irreducible.

By assumption, there is a unique isotropic subspace $F_{a_i}$ such that $\dim(\Lambda\cap F_{a_i})\geq i$ for all $\Lambda\in X$. We use induction on $j$. If $j=1$, then it suffices to show $Y=\mathbb{P}(F_{a_i}^\perp)\cap Q$. Take a general point $p\in\mathbb{P}(F_{a_i})$. Set $X_p:=\{\Lambda\in X|p\in\Lambda\}$. The class of $X_p$ is given by $m\sigma_{a';b'}$ where $a'_1=1$, $a'_{i'}=a_{i'-1}+1$ if $2\leq i'\leq i$, $a'_{i'}=a_{i'}$ if $i'>i$, $b'_{j'}=b_{j'}$ for $1\leq j'\leq k$. Let $Z$ be the variety swept out by $\mathbb{P}^{k-1}$ parametrized by $X_p$. Let $G_{a_{i}}$ and $G_{a_i+1}$ be two general isotropic subspaces of dimension $a_i$ and $a_i+1$ respectively. Let $\Sigma_{a_i}:=\{\Lambda\in OG(k,n)|\dim(\Lambda\cap G_{a_i})\geq 1\}$ and $\Sigma_{a_i+1}:=\{\Lambda\in OG(k,n)|\dim(\Lambda\cap G_{a_i+1})\geq 1\}$ be the corresponding Schubert varieties. By intersecting $X$ and $X_p$ with the Schubert varieties $\Sigma_{a_i}$ and $\Sigma_{a_i+1}$, we obtain that $\dim(Y)=\dim(Z)=n-a_i-2$. Since $Y$ is irreducible and by the construction it is necessary that $Z\subset Y$, we conclude that $Y=Z$. Also notice that for every $\Lambda\in X_p$, $\Lambda\subset p^\perp$ and hence $Z\subset p^\perp$. Since $Y=Z\subset Y\cap p^\perp$, $Y$ must be contained in $p^\perp$ for every $p\in F_{a_i}$. Therefore $Y\subset \mathbb{P}(F_{a_i}^\perp)\cap Q=Q_{n-a_i}^{a_i}$, where $Q_{n-a_i}^{a_i}$ is a sub-quadric of corank $a_i$ and dimension $n-a_i-2$. In particular, $Q_{n-a_i}^{a_i}$ is irreducible and therefore $Y=Q_{n-a_i}^{a_i}$ as desired.

Now assume $j\geq 2$. Let $q$ be a general point in $Y$ and set $X_q:=\{\Lambda\in X|q\in\Lambda\}$. Then the class of $X_q$ is given by $m\sigma_{a'';b''}$, where $a''_1=1$, $a''_{i''+1}=a_{i''}+1$ if $a_{i''}\leq b_1$, $a''_{i''+1}=a_{i''}$ if $a_{i''}>b_1$, $b''_{j''}=b_{j''+1}$ for $1\leq j''\leq k-s-1$. In particular, $a''_{i+1}=b''_{j-1}$ and $a''_{i+1}$ satisfies the condition in Theorem \ref{multiindex}. Set $F_{a_i}^q=$span$(q,q^\perp\cap F_{a_i})$. Then for all $\Lambda\in X_q$, $\Lambda\subset q^\perp$ and therefore
\begin{eqnarray}
\dim(\Lambda\cap F_{a_i}^q)&=&\dim(\Lambda\cap \text{span}(q,q^\perp\cap F_{a_i}))\nonumber\\
&\geq&1+\dim(\Lambda\cap F_{a_i})\nonumber\\
&\geq&1+i\nonumber
\end{eqnarray}
By induction, we assume $\dim(\Lambda\cap (F_{a_i}^q)^\perp)\geq k-j+2$ for all $\Lambda\in X_q$. Since $F_{a_i}^p$ intersects $F_{a_i}$ in a codimension $1$ subspace, we have all $\Lambda\in X_q$,
\begin{eqnarray}
\dim(\Lambda\cap F_{a_i}^\perp)&=&\dim(\Lambda)-\dim(F_{a_i})+\dim(\Lambda^\perp\cap F_{a_i})\nonumber\\
&\geq&\dim(\Lambda)-\dim(F_{a_i})+\dim(\Lambda^\perp\cap F^p_{a_i})-1\nonumber\\
&\geq&\dim(\Lambda\cap (F_{a_i}^p)^\perp)-1\nonumber\\
&\geq&k-j+1.\nonumber
\end{eqnarray}
Since $X_q$ covers $X$ as we varying $q\in Y$, $\dim(\Lambda\cap F_{a_i}^\perp)\geq k-j+1$ for all $\Lambda\in X$. We conclude that the statement is true by induction.
\end{proof}

\begin{Cor}
Let $\sigma_{a;b}$ be a Schubert class in $OG(k,n)$. If all essential $a_i$ satisfies the condition in Theorem \ref{multiindex}, and for every essential $b_j$, $b_j=a_i$ for some $i$, then $\sigma_{a;b}$ is multi-rigid.
\end{Cor}
\begin{proof}
Let $X$ be an irreducible representative of $m\sigma_{a;b}$. By Proposition \ref{bmulti}, there is a unique isotropic subspace corresponding to each essential $a_i$. A similar proof as in Lemma \ref{coincide} shows that those isotropic subspaces form a partial flag. It is then easy to see that $X$ is the Schubert variety defined by this partial flag.
\end{proof}

\section{Applications to symplectic Grassmannians/partial flag varieties}\label{type C}
\subsection{Symplectic Grassmannians}Let $V$ be a complex vector space of dimension $n$. Assume $n=2m$ is even. Let $\omega$ be a non-degenerate, skew-symmetric bilinear form on $V$. A subspace $W$ is called isotropic if for every $v,w\in W$, $\omega(v,w)=0$. The symplectic Grassmannian $SG(k,n)$ parametrizes all $k$-dimensional isotropic subspaces of $V$ with respect to $\omega$. It is a homogeneous variety $G/P$ with $G=Sp(n)$ and $P$ a maximal parabolic subgroup of $G$. 

The Schubert varieties for $SG(k,n)$ can be defined similarly using an isotropic flag.
\begin{Def}
A Schubert index for $SG(k,n)$ consists of two increasing sequences of non-negative integers
$$1\leq a_1<...<a_s\leq m,$$
$$0\leq b_1<...<b_{k-s}\leq m-1,$$
such that $a_i-b_j\neq 1$ for any $1\leq i\leq s$ and $1\leq j\leq k-s$.

Given an isotropic flag
$$F_{a_1}\subset...\subset F_{a_s}\subset F_{b_{k-s}}^\perp\subset...\subset F_{b_1}^\perp,$$
the Schubert variety $\Sigma_{a;b}(F_\bullet)$ is defined to be the Zariski closure of the following locus:
$$\Sigma^\circ_{a;b}(F_\bullet):=\{\Lambda\in SG(k,n)|\dim(\Lambda\cap F_{a_i})=i,\dim(\Lambda\cap F_{b_j}^\perp)=k-j+1\}.$$
\end{Def}

Same as in the previous sections, we ask for the rigidity on each sub-index:
 \begin{Def}
Let $\sigma_{a;b}$ be a Schubert class in $SG(k,n)$. A sub-index $a_i$ is called {\em rigid} if for every subvariety $X$ of $SG(k,n)$ representing $\sigma_{a;b}$, there exists a unique isotropic subspace $F_{a_i}$ of dimension $a_i$ such that $$\dim(\Lambda\cap F_{a_i})\geq i,\ \ \ \forall\Lambda\in X.$$
Similarly, a sub-index $b_j$ is called {\em rigid} if for every subvariety $X$ representing $\sigma_{a;b}$, there exists a unique isotropic subspace $F_{b_j}$ of dimension $b_j$ such that $$\dim(\Lambda\cap F_{b_j}^\perp)\geq k-j+1\ \ \ \forall\Lambda\in X.$$
\end{Def}

Consider the natural inclusion $SG(k,n)\hookrightarrow G(k,n)$. The push-forward of a Schubert class $\sigma_{a;b}$ under $i$ can be computed using Coskun's geometric branching rule \cite{Coskun2012}. In particular, if $a_i=i$ and either $a_{i+1}-a_i\geq 3$ or $i=s$ and $n\geq 2k+2$, then the push-forward class is a finite sum $\sum_\lambda c_\lambda\sigma_\lambda$, where $c_\lambda$ are positive integers and $\sigma_\lambda$ are Schubert classes in $G(k,n)$ such that $\lambda_i=1$ for all $\lambda$, and there exists at least one $\lambda$ such that $\lambda_{i+1}-\lambda_i\geq 3$. Now combine with Theorem \ref{rigid index in mix}, we obtain the multi-rigidity (and hence rigidity) of such Schubert classes:
\begin{Cor}\label{SG1}
Let $\sigma_{a;b}$ be a Schubert class in $SG(k,n)$. If $a_i=i$ and
\begin{itemize}
\item either $a_{i+1}-a_i\geq 3$, or
\item $i=s$ and $n\geq 2k+2$
\end{itemize}
then the sub-index $a_i$ is rigid.
\end{Cor}

\begin{Cor}
The Schubert class $\sigma_{a;b}$, where $a=(1,...,i)$ and $b=(i,i+1,...,k-1)$, $1\leq i\leq k$ is rigid for $SG(k,n)$, $n\geq 2k+2$.
\end{Cor}
\begin{proof}
By Corollary \ref{SG1}, for every representative $X$ of $\sigma_{a;b}$, there exists a unique isotropic subspace $F_i$ of dimension $i$ that is contained in every $\Lambda$ in $X$. The inclusion $F_i\subset\Lambda$ implies $\Lambda\subset F_i^\perp$, and therefore $X$ is contained in the Schubert variety $\Sigma_{a;b}:=\{\Lambda\in SG(k,n)|F_i\subset \Lambda\subset F_i^\perp\}$. By comparing the dimension, we conclude that $X$ is the Schubert variety $\Sigma_{a;b}$. 
\end{proof}

\subsection{Symplectic partial flag varieties}Let $V$ and $\omega$ be defined as before. The symplectic partial flag variety $SF(d_1,...,d_k;n)$ parametrizes $k$-step isotropic partial flags $\Lambda_1\subset...\subset \Lambda_k$. When $k=1$, we obtain the symplectic Grassmannian $SG(d_1,n)$.

Fix an isotropic flag 
$F_1\subset ...\subset F_{m}\subset F_{m-1}^\perp...\subset F_1^\perp\subset F_0^\perp=V$ and set
$$\mu_{i,t}:=\#\{c|a_c\leq a_i,\alpha_c\leq t\},$$
$$\nu_{j,t}:=\#\{d|\alpha_d\leq t\}+\#\{e|b_e\geq b_j,\beta_e\leq t\}.$$
The Schubert variety $\Sigma_{a^\alpha;b^\beta}$ is then defined to be the Zariski closure of the following locus:
\begin{eqnarray}
\Sigma_{a^\alpha;b^\beta}:=\{(\Lambda_1,...,\Lambda_k)\in SF(d_1,...,d_k;n)&|&\dim(F_{a_i}\cap \Lambda_t)= \mu_{i,t}, \alpha_i\leq t,\nonumber\\
& &\dim(F^\perp_{b_j}\cap \Lambda_t)=\nu_{j,t},\beta_j\leq t\}.\nonumber
\end{eqnarray}
Using a similar argument for orthogonal partial flag varieties, we can deduce the rigidity in $SF(k,n)$ from the rigidity in $SG(k,n)$.
\begin{Prop}
Let $\sigma_{a^\alpha;b^\beta}$ be a Schubert class in $SF(d_1,...,d_k;n)$. Let $X$ be a representative of $\sigma_{a^\alpha;b^\beta}$. Let $F_{a_i}$ and $F_{b_j}$ be isotropic subspaces of dimension $a_i$ and $b_j$ respectively. If for some $\alpha_i\leq t\leq k$, $$\dim(F_{a_i}\cap\Lambda_t)\geq \mu_{i,t}\text{ for all }(\Lambda_1,...,\Lambda_k)\in X,$$ then the same inequailty holds for all $\alpha_i\leq t'\leq k$.

Similarly, if for some $\beta_j\leq t\leq k$, $$\dim(F_{b_j}^\perp\cap\Lambda_t)\geq \nu_{j,t}\text{ for all }(\Lambda_1,...,\Lambda_k)\in X,$$ then the same inequality holds for all $\beta_j\leq t'\leq k$.
\end{Prop}

\bibliographystyle{plain}

\end{document}